\documentclass{article}[11pt]

\usepackage{amsmath}   
\usepackage{amssymb}
\usepackage{accents}
\usepackage{amsthm}
\usepackage{tikz-cd} 


\newtheorem{Th}{Theorem}[section] 
\newtheorem{Prop}{Proposition}[section]   
\newtheorem{Lem}{Lemma}[section]   
\newtheorem{Coro}{Corollary}[section]   
\newtheorem{Def}{Definition}  
\newtheorem{Rem}{Remark}[section]

\newcommand{\R}{\mathbb{R}}
\newcommand{\Z}{\mathbb{Z}}
\newcommand{\C}{\mathbb{C}}

\newcommand{\N}{\mathbb{N}}

\newcommand{\bbS}{\mathbb{S}}
\newcommand{\s}{{\rm S}}
\newcommand{\Sz}{\mathcal{S}}
\newcommand{\x}{\langle x\rangle}
\newcommand{\y}{\langle y\rangle}
\newcommand{\z}{\langle z\rangle}
\newcommand{\Lap}{\Delta}
\newcommand{\A}{{\mathcal A}}

\newcommand{\AH}{{\mathbb A}}

\newcommand{\F}{{\mathcal F}}

\newcommand{\re}{\text{\rm Re}}

\newcommand{\dt}[1]{\accentset{\mbox{\bfseries .}}{#1}}
\newcommand{\intr}[1]{\accentset{\circ}{#1}}

\newcommand{\e}{{\varepsilon}}

\begin{document}

\title{Spatial asymptotics and equilibria of heat flow on $\R^d$}   
 
\author{Robert McOwen and Peter Topalov\footnote{P.T. partially supported by the Simons Foudation, Award \#526907.}} 

\maketitle

\begin{abstract} 
We prove that the heat equation on $\R^d$ is well-posed in certain spaces of functions allowing spatial 
asymptotic expansions as $|x|\to\infty$ of any a priori given order. In fact, we show that the Laplacian on 
such function spaces generates an analytic semigroup of angle $\pi/2$ with polynomial growth as $t\to\infty$. 
Generically, a large class of nonlinear heat flows have equilibrium solutions with spatial asymptotics of the considered type. 
We provide a simple nonlinear model that features global in time existence with such asymptotics at spatial infinity.
\end{abstract}


\section{Introduction}\label{sec:introduction}
Consider the initial-value problem for the heat equation on $[0,\infty)\times\R^d$:
\begin{equation}\label{eq:heat}
\left\{
\begin{array}{l}
u_t=\Lap u, \quad\hbox{for $t>0$, $x\in\R^d$,} \\
u|_{t=0}=v, \quad\hbox{for $x\in\R^d$,}
\end{array}
\right.
\end{equation}
where $v\!\in\! CB(\R^d)$, i.e.\ the continuous, bounded functions on $\R^d$. 
It is well-known (see e.g.\ \cite{McOwen,Shubin}) that the unique solution $u\in CB([0,\infty)\times\R^d)$ of 
\eqref{eq:heat} is given by
\begin{equation}\label{eq:Gaussian}
u(x,t)=\big(S(t)v\big)(x):=\frac{1}{(4\pi t)^{d/2}} \int_{\R^d} e^{-\frac{|x-y|^2}{4t}} v(y)\,dy,
\end{equation}
and that $u\in C^\infty((0,\infty)\times\R^d)$. However, if $v$ has certain asymptotic properties as $|x|\!\to\!\infty$,
we are interested to know whether these properies are transmitted to $u(x,t)$.
In particular, suppose that $v$ has a partial asymptotic expansion
\begin{equation}\label{eq:v-asymptotic}
v(x)=b_0(\theta)+\frac{b_1(\theta)}{r}+\cdots+\frac{b_N(\theta)}{r^N}+o\left(\frac{1}{r^N}\right)
\quad\hbox{as $|x|\to\infty$},
\end{equation}
where $r:=|x|$, $\theta:=x/|x|$, and the coefficients $b_k$ are continuous functions on the unit sphere $\s^{d-1}$.
We want to conclude that $u(x,t)$ has a similar expansion with coefficients depending on $t\ge 0$:
\begin{equation}\label{eq:u-asymptotic}
u(x,t)=a_0(\theta,t)+\frac{a_1(\theta,t)}{r}+\cdots+\frac{a_N(\theta,t)}{r^N}+o\left(\frac{1}{r^N}\right)
\quad\hbox{as $|x|\to\infty$},
\end{equation}
such that for any $\theta\in S^{d-1}$ we have that $a_k(\theta,t)\to b_k(\theta)$ as $t\to 0$.
Of course, part of the challenge is to determine the conditions to impose on the coefficients $b_k(\theta)$, $a_k(\theta,t)$, 
and also how to handle the remainder term $o(1/r^{N})$. In fact, we want to do more. We want to consider initial 
conditions $v$ belonging to a Banach space $X$ of functions with  asymptotic conditions as $|x|\to\infty$, and show that 
$\{S(t)\}_{t\ge 0}$ is an analytic semigroup on $X$. The semigroup $S(t)$ also plays a role in the study of 
nonlinear equations of the form
\begin{equation}\label{eq:nonlinearheat}
u_t=\Lap u+F(u).
\end{equation}
In fact, such equations provide a motivation for considering asymptotics of the form \eqref{eq:v-asymptotic} since as we 
shall see below (cf.\ Theorem \ref{th:semilinear-equilibrium}), 
a large class of equations of the form \eqref{eq:nonlinearheat}, {\em generically}, have {\em equilibrium solutions} with 
asymptotics of the form \eqref{eq:v-asymptotic}. 


Let us describe the function spaces that we will consider here; they were studied in more detail in \cite{McOwenTopalov2}, 
but their relevant properties are summarized in Appendix \ref{sec:appendix_properties}.
In the following, let $\N$ denote the natural numbers, $\Z_{\ge 0}$ denote the nonnegative integers, and 
$C^\infty_c(\R^d)$ denote the smooth functions with compact support.  Moreover, we let  $\x:=\sqrt{1+|x|^2}$ and 
$\chi(r)$ be a $C^\infty$-function for $r\ge 0$ such that $\chi(r)=0$ for $0\le r\le 1$ and $\chi(r)=1$ for $r\ge 2$.

\medskip

\noindent{\em Definition of $H_\delta^{m,p}(\R^d)$.}
For $1\le p<\infty$, $m\in\Z_{\ge 0}$, and $\delta\in\R$, define the weighted Sobolev space $H_\delta^{m,p}(\R^d)$ to be 
the closure of $C^\infty_c(\R^d)$ in the norm
\begin{equation}\label{eq:H-norm}
\|f\|_{H^{m,p}_\delta}=\sum_{|\alpha|\le m} \|\x^\delta\partial^\alpha f\|_{L^p}.
\end{equation}
(Here we use multi-index notation, i.e.\ $\partial^\alpha=\partial_1^{\alpha_1}\cdots\partial_d^{\alpha_d}$ where 
$\partial_j:=\partial/\partial x_j$ 
and  $|\alpha|=\alpha_1+\cdots\alpha_d$.) For $m=0$, we write $L^p_\delta(\R^d)$ instead of $H^{0,p}_\delta(\R^d)$, 
and for $\delta=0$ we write $H^{m,p}(\R^d)$ instead of $H^{m,p}_0(\R^d)$.
We sometimes write $H^{m,p}_\delta$ instead of $H^{m,p}_\delta(\R^d)$.
 
\medskip

\noindent{\em Definition of $\AH^{m,p}_N$ and $\AH^{m,p}_{n,N}$.} 
For $m,N\in\Z_{\ge 0}$ with $m>d/p$ and $1\le p<\infty$, let us denote by $N^*\in\Z_{\ge 0}$ the 
integer\footnote{We always have $N^*\ge N$, but  $N^*=N$ if and only if $d<p$; this was the case in \cite{McOwenTopalov1}, 
where $d=1$ and $p>1$.} satisfying 
\begin{equation}\label{def:N*}
N-1<N^*-d/p\le N,
\end{equation}
and define $\AH^{m,p}_{N}$ to be the space of functions of the form
\begin{equation}\label{AH-expansion1} 
v(x)=\chi(r)\left(b_0(\theta)+\cdots+\frac{b_{N^*}(\theta)}{r^{N^*}}\right)+f(x)
\end{equation}
where\footnote{$H^{l,p}(\s^{d-1})$ denotes the $L^p$-based Sobolev space on $\s^{d-1}$ of regularity $l\ge 0$.}
\begin{equation}\label{eq:b-term}
b(x)=\chi(r)\left(b_0(\theta)+\cdots+\frac{b_{N^*}(\theta)}{r^{N^*}}\right)\,\,
\text{with $b_k\in H^{m+1+N^*-k,p}(\s^{d-1})$}
\end{equation}
and
\[
f\in H^{m,p}_N(\R^d). 
\]
We shall call $b$ the {\em asymptotic function}, $b_k$ the  {\em asymptotic coefficients}, and $f$ the {\it remainder function}.
The restriction \eqref{def:N*} guarantees that 
$\chi(r)/r^{N^*+1}\in H_N^{m,p}(\R^d)$ but $\chi(r)/r^{N^*}\not\in H_N^{m,p}(\R^d)$, 
so the representation \eqref{AH-expansion1} is unique.
Finally, note that the decreasing regularity of the asymptotic coefficients in \eqref{eq:b-term} plays an
essential role for proving the results in this paper.

Let us explain why these assumptions guarantee that $v\in\AH^{m,p}_{N}$ is of the desired form
\begin{equation}\label{v-asymptotics}
v(x)=\left(b_0(\theta)+\cdots + \frac{b_{N}(\theta)}{r^{N}}\right)+o\left(\frac{1}{r^N}\right)
\quad\hbox{as $|x|\to\infty$}.
\end{equation}
We have assumed $m>d/p$ because then $f\in H^{m,p}_N(\R^d)$ implies $f$ is continuous and satisfies $f(x)=o(|x|^{-N})$ as 
$|x|\to\infty$ (cf.\ Proposition \ref{pr:H-properties}(a) in Appendix A).
Moreover, $N^*\ge N$, so the extra asymptotic terms in  \eqref{eq:b-term} can be put into the $o(1/r^N)$ term in 
\eqref{v-asymptotics}. 
The function space $\AH^{m,p}_{N}$ becomes a Banach space under the norm
\begin{equation}\label{def:AH-norm}
\|v\|_{{\AH}_{N}^{m,p}}=
\sum_{k=0}^{ N^*} \|b_k\|_{H^{m+1+N^*-k,p}}+\| f\|_{H_N^{m,p}}.
\end{equation}
For an integer $n$ satisfying $0\le n\le N$, we denote by $\AH^{m,p}_{n,N}$ the closed subspace of $\AH^{m,p}_N$ for 
which $b_0=\cdots=b_{n-1}=0$. Note that for $m>d/p$ the space $\AH^{m,p}_{n,N}$ is a Banach algebra 
(cf.\ Proposition \ref{pr:A-properties}(e)).

\medskip

We prove the following theorem.

\begin{Th}\label{th:main}
Assume $1<p<\infty$, $N\in \Z_{\ge 0}$, and $m\in\Z_{\ge 0}$ satisfies $m>d/p$.
Then  \eqref{eq:Gaussian} defines a strongly continuous semigroup on 
$\AH^{m,p}_N$ that satisfies
\begin{equation}\label{est:S(t)onAH}
\|S(t)v\|_{\AH^{m,p}_N}\le C\,(1+t)^{\mu}\|v\|_{\AH^{m,p}_N}\quad\hbox{for }\,\,0<t<\infty,\ v\in \AH^{m,p}_N, 
\end{equation}
where $\mu:=\big(N+N^*+2\big)/2>0$, and  is analytic  of angle $\pi/2$. The generator is $\Lap$ considered as 
an unbounded operator on $\AH^{m,p}_N$ with a domain $D$ that contains $\AH^{m+2,p}_N$. 
Moreover, if $N\ge 2$ then the two leading asymptotic coefficients $a_0(t)$ and $a_1(t)$ in \eqref{eq:u-asymptotic} are 
independent of $t\ge 0$.
\end{Th}

\begin{Rem}
In fact, we prove that the analog of \eqref{est:S(t)onAH} holds in 
$\big\{z\in\C\setminus\{0\} : |\arg z|<\pi/2-\epsilon\big\}$ for any $\epsilon\in(0,\pi/2)$ 
with a constant depending on $\epsilon$ but independent on $m\ge 0$ 
(see Theorem \ref{th:SonH} and Theorem \ref{th:SonAH}).
\end{Rem}

Note that, by Theorem \ref{th:main}, the heat flow $\{S(t)\}_{t\ge 0}$ on $\AH^{m,p}_N$ preserves the regularity 
of the asymptotic coefficients of the initial data $v\in\AH^{m,p}_N$ in \eqref{eq:heat}, i.e., 
if we have in \eqref{AH-expansion1} that $b_k\in H^{m+1+N^*-k,p}(\s^{d-1})$, $0\le k\le N^*$, then for any $t\ge 0$ 
we have that $a_k(t)\in H^{m+1+N^*-k,p}(\s^{d-1})$,  $0\le k\le N^*$, for the asymptotic coefficients of
the solution $u(t)\equiv S(t)v\in\AH^{m,p}_N$. 
Moreover, since $a_0(t)=b_0$, we see that there is no smoothing in the leading asymptotic term of the solution $u(t)$ 
for $t>0$. (In fact, more is true: One can see from the proof of Theorem \ref{th:main} that generically
a similar non-smoothing property holds also for the rest of the the asymptotic coefficients $a_k(t)$, 
$1\le k\le N^*$, of the solution.)
That is in striking contrast with the fact that the solution $u(t)\equiv S(t)v$ of \eqref{eq:heat} is itself smooth for $t>0$, 
i.e., $u(t)\in C^\infty(\R^d)$ for $t>0$. We will call this property of the linear heat flow  {\em asymptotic non-smoothing}.

\medskip

In Section \ref{sec:weighted_spaces} we prove Theorem \ref{th:main} with $H^{m,p}_\delta(\R^d)$ for 
$\delta\in\R$ in place of $\AH^{m,p}_N$. This result is extended to $\AH^{m,p}_N$ in Section \ref{sec:asymptotic_spaces}. 
 
\medskip 
 
As mentioned above, in Section \ref{sec:application-nonlinear} we apply Theorem \ref{th:main} to study 
semilinear heat equations.
Rather than considering general conditions under which we have local and global existence in time, we will
focus on a case where we have global existence in time and 
 equilibrium solutions with nontrivial asymptotics of the type considered in this paper. In particular, we consider 
the specific semilinear equation
\begin{equation}\label{intro-semilinear1}
\left\{
\begin{array}{l}
u_t=\Lap u+\varphi-\psi u^3\quad\text{for}\,\,t>0,\,x\in\R^3,\\
u|_{t=0}=v,
\end{array}
\right.
\end{equation}
where $\varphi$ and $\psi\ge 0$ are functions in the Schwartz space $\Sz(\R^3)$.
For $v\in \AH^{m+2,p}_{1,N}$ we want a solution of \eqref{intro-semilinear1} satisfying $u(t)\in  \AH^{m+2,p}_{1,N}$ for $t>0$.
We will prove the following:
\begin{Th}\label{th:semilinear-global} 
For  $p>3$, $N\ge 1$, $m>3/p$, and $v\in \AH_{1,N}^{m+2,p}$, equation \eqref{intro-semilinear1} has a unique global solution 
$u\in C([0,\infty),\AH_{1,N}^{m+2,p})\cap C^1([0,\infty),\AH_{1,N}^{m,p})$.
\end{Th}

\noindent {\em Equilibrium} (or, equivalently, {\em stationary}) solutions of \eqref{intro-semilinear1} are solutions
of the semilinear elliptic equation
\begin{equation}\label{eq:semilinear-equilibrium}
\Delta u=\psi\, u^3-\varphi\quad\hbox{for}\quad x\in\R^3.
\end{equation}
We show in Section \ref{sec:application-nonlinear} that the semilinear heat equation \eqref{intro-semilinear1} has a unique 
equilibrium solution with an asymptotic expansion at infinity of the type considered in this paper, and this expansion is 
generically non-vanishing. More specifically, we will prove the following:

\begin{Th}\label{th:semilinear-equilibrium} 
Assume that $p>3$ and let $m,N\ge 1$ be integer. Then, for any given $\psi\in\Sz(\R^3)$, $\psi\ge 0$, 
and for any $\varphi\in\Sz(\R^3)$  there exists a unique equilibrium solution $u_*\in\AH^{m+2,p}_{1,N}$ of \eqref{intro-semilinear1},
\begin{equation}\label{eq:u_*-asymptotics}
u_*(x)=\frac{a_1(\theta)}{r}+\cdots+\frac{a_N(\theta)}{r^N}+o\Big(\frac{1}{r^N}\Big)\quad\text{\rm as}\quad|x|\to\infty,
\end{equation}
where the $a_k\in C^\infty(S^2)$ are eigenfunctions for the Laplace operator on $S^2$.
Moreover, there exists an open and dense set $\mathcal{N}$ in $\Sz$ such that for any $\varphi\in\mathcal{N}$ and for any
$1\le k\le N$ the coefficient $a_k$ in \eqref{eq:u_*-asymptotics} does not vanish.
\end{Th}

\noindent
The existence of an equilibrium solution of \eqref{intro-semilinear1} in $\AH^{m+2,p}_{1,N}$ provides motivation for the study of the time evolution in this function space, which is the subject of this paper. However, 
the question whether these equilibrium solutions are stable in $\AH^{m,p}_N$ and what kind of stability pertains is open.

\medskip

Much of what we have done in this paper  extends to more general elliptic operators than the Laplacian in the Euclidean space. 
But we have chosen to focus on the heat equation in the Euclidean space in order to utilize the classical solution formula 
\eqref{eq:Gaussian} and make proofs as accessible as possible. The technique developed in this paper
allows us to prove that the Navier-Stokes equation is locally in time well-posed in the asymptotic spaces (cf.\ \cite{McOwenTopalov5}). 
We refer to \cite{McOwenTopalov3,McOwenTopalov4} for similar results related to the Euler equation.
It is worth mentioning that, generically, the solutions of the Euler and Navier-Stokes equations with fast decaying initial data
develop non-trivial spatial asymptotics (see e.g. \cite[Theorem 1.2]{McOwenTopalov4}).
This together with Theorem \ref{th:semilinear-equilibrium} above implies that the spatial asymptotics considered in the present paper
appear as a natural phenomenon in various equations describing important physical models.

\section{The semigroup on $H^{m,p}_\delta(\R^d)$}\label{sec:weighted_spaces}

In this section we prove preliminary results on the heat semigroup on weighted Sobolev spaces $H^{m,p}_\delta(\R^d)$
that are needed for the sequel. 


Recall (cf.\ \cite{Pazy}, \cite{ReedSimon}) that a strongly continuous semigroup $\{S(t)\}_{t\ge 0}$ of bounded linear operators 
on a Banach space $X$ is {\em analytic of angle $\vartheta\in (0,\pi/2]$} if the map $t\mapsto S(t)$ can be extended
to the complex sector   
\[
\bbS_\vartheta:=\{0\}\cup\{z\in\C: |\arg\,z|<\vartheta\}
\] 
such that i) the map $z\mapsto S(z)$ is analytic from ${\bbS}_\vartheta\backslash\{0\}$ to ${\mathcal L}(X)$, the bounded linear 
operators on $X$, ii) $S(z)$ is a semigroup, i.e.\ $S(z_1+z_2)=S(z_1)S(z_2)$ for $z_1,z_2\in {\bbS}_\vartheta$, and iii) $S(z)$ 
converges strongly to the identity operator on $X$ as $z\to 0$ in any subsector 
$\bbS_{\vartheta-\epsilon}$ with $0<\epsilon<\vartheta$.

In our case, we let $S(t)$ be defined by \eqref{eq:Gaussian} for $t>0$ and let $S(0)=I$. We can replace $t$ by $z\in {\bbS}_{\pi/2}$ 
in \eqref{eq:Gaussian} and the resultant operator $S(z)$ is still well-defined.  In fact, it is well-known (cf.\ e.g.\ \cite[IX.1]{Kato}) 
that for $1<p<\infty$, $S(t)$ is a contraction semigroup on $L^p(\R^d)$ that defines an analytic semigroup  of angle $\pi/2$ whose 
generator is $\Lap$, viewed as an unbounded operator on $L^p(\R^d)$ with domain $H^{2,p}(\R^d)$. 
We will prove the following generalization:

\begin{Th}\label{th:SonH} 
For $1< p<\infty$ and for any $m\in\Z_{\ge 0}$ and $\delta\in\R$, \eqref{eq:Gaussian} defines a strongly continuous semigroup 
of bounded operators $\{S(t)\}_{t\ge 0}$ on $H_\delta^{m,p}(\R^d)$. The generator of the semigroup is $\Lap$, 
viewed as an unbounded operator on $H^{m,p}_\delta(\R^d)$ with domain $D_{H^{m,p}_\delta}(\Lap)=H^{m+2,p}_\delta(\R^d)$.
Moreover, the semigroup is analytic of angle $\pi/2$ and  for any $\epsilon\in(0,\pi/2)$ there exists a positive constant 
$C\equiv C_{p,\delta,\epsilon}>0$ such that for any $m\in\Z_{\ge 0}$ we have
\begin{equation}\label{eq:S-estimate}
\big\|S(z)\big\|_{\mathcal{L}(H^{m,p}_\delta)}\le C\,\big(1+|z|\big)^{|\delta|/2}
\end{equation}
for all $z\in\bbS_{\pi/2-\epsilon}\setminus\{0\}$.
\end{Th}

\begin{Rem}
Note that the polynomial estimate \eqref{eq:S-estimate} is stronger than the exponential estimate 
$\big\|S(z)\big\|_{\mathcal{L}(H^{m,p}_\delta)}\le C_{\epsilon,\omega}\,e^{\omega|z|}$ for any $\omega,\epsilon>0$
suggested by the general theory of semigroups. Also note that $\delta<0$ in $H^{m,p}_\delta$ allows functions
that grow as $|x|\to\infty$, but not fast enough to cause finite-time blow-up like occurs in \cite{Tychonoff}, \cite{RobinsonRodriguez-Bernal}.
Finally, note that Propositions 3.1 and 3.2 (iii) in \cite{AMR-B} imply Theorem \ref{th:SonH} in the case when $m=0$,
but without specifying the angle of analytic extension as $\pi/2$ and with \eqref{eq:S-estimate} established
only for real $z=t>0$.
\end{Rem}

To prove Theorem \ref{th:SonH}, we first show that $\Lap$ with domain $H^{m+2,p}_\delta(\R^d)$ generates 
an analytic semigroup of angle $\vartheta=\pi/2$ on $H_\delta^{m,p}(\R^d)$ denoted by $e^{z\Lap}$. 
Our second step is to prove that $S(t)=e^{t\Lap}$ for all $t\ge 0$.
We begin with the following simple lemma. We include the proof for convenience.

\begin{Lem}\label{le:J_delta-H} 
For $1\le p<\infty$ and for any $m\in\Z_{\ge 0}$ and $\delta\in\R$ the multiplication operator $J_\delta f:=\x^\delta f$ 
defines an isomorphism
\begin{equation}\label{eq:J_delta}
J_\delta : H_\delta^{m,p}(\R^d)\to H^{m,p}(\R^d).
\end{equation}
In fact, \eqref{eq:J_delta} is an isometry for $m=0$.
\end{Lem}

\proof
Clearly $J_\delta$ is injective. For $m=0$, we have 
\[
\|J_\delta f\|_{L^p}^p=\int |\x^\delta f(x)|^p\,dx=\|f\|_{L_\delta^p}^p,\]
so \eqref{eq:J_delta} is an isometry, and it is surjective since any $f\in L^p(\R^d)$ can be written as $f=J_\delta g$,
where $g=\x^{-\delta}f\in L^p_\delta(\R^d)$. Now  consider $m>0$.
For any multi-index $0\le|\beta|\le m$ we have $\big|\partial^\beta(\x^\delta)\big|\le C\x^{\delta-|\beta|}\le C\x^\delta$ with a constant
$C>0$ depending on $m$, which implies that for any multi-index $0\le|\alpha|\le m$,
\begin{equation*}
\big|\partial^\alpha\left(\x^\delta f\right)\big|\le C_1\,\sum_{|\beta|\le |\alpha|}\x^\delta\,\big|\partial^\beta f(x)\big|
\end{equation*}
with a constant $C_1>0$ depending on $m$.
This and the Jensen inequality then imply
\[
\|J_\delta f\|_{H^{m,p}}=\sum_{|\alpha|\le m}\big\|\partial^\alpha\big(\x^\delta f\big)\big\|_{L^p}
\le C_2 \sum_{|\beta|\le m}\big\| \x^\delta\partial^\beta f\big\|_{L^p}=C_2\,\|f\|_{H^{m,p}_\delta}
\]
for some constant $C_2>0$ depending on $m>0$.
This shows that \eqref{eq:J_delta} is bounded, and it is also surjective since any $f\in H^{m,p}(\R^d)$ can be written as 
$f=J_\delta g$, where $g=\x^{-\delta}f\in H^{m,p}_\delta(\R^d)$. By the inverse mapping theorem, \eqref{eq:J_delta} is 
an isomorphism.
\endproof

Let us now study $\Lap$ as the generator of an analytic semigroup on $H^{m,p}_\delta(\R^d)$.
Recall (cf.\ \cite[\S\,2.5]{Pazy}, \cite{ReedSimon}) that an unbounded operator $\Lambda$ on a 
Banach space $X$ generates an analytic semigroup  of angle $\vartheta$  when $\Lambda$ is a closed operator with dense domain $D$, 
its resolvent set $\rho(\Lambda)$ contains the sector
\begin{equation}\label{def:S_theta,omega}
\Sigma_{\vartheta,\omega}:=\big\{\lambda\in\C:\lambda\not=\omega,\,|\arg(\lambda-\omega)|<\vartheta+\pi/2\big\},
\,\,\text{for some}\,\,\omega\in\R,
\end{equation}
and for every $\epsilon\in (0,\vartheta)$ there exists  $C_{\epsilon,\omega}>0$ so that the resolvent operator 
$R_\Lambda(\lambda)=(\lambda I-\Lambda)^{-1}$ satisfies
\begin{equation}\label{eq:Inf-Gen-Est}
\big\|R_\Lambda(\lambda)\big\|_{\mathcal{L}(X)}\le\frac{C_{\epsilon,\omega}}{|\lambda-\omega|}
\quad\hbox{for all  $\lambda\in\Sigma_{\vartheta-\epsilon,\omega}$}.
\end{equation}
In this case, the analytic semigroup $e^{z\Lambda}$ is defined by
\begin{subequations} \label{e^zGamma}
\begin{equation}\label{def:e^zGamma}
e^{z\Lambda}:=\frac{1}{2\pi i}\int_{\omega+\Gamma} e^{z\lambda} R_\Lambda(\lambda)\,d\lambda
\quad\hbox{for}\,\,z\in{\bbS}_\vartheta\backslash\{0\},
\end{equation}
where $\omega+\Gamma$ is the shift by $\omega$ of a curve $\Gamma$ defined as follows: choose $\rho_0>0$ and 
$\phi\in(\pi/2,\pi/2+\vartheta)$ such that $|z|>\rho_0$ and $|\arg z|<\phi-\pi/2$; then let 
\begin{equation}\label{def:gamma}
\Gamma:=\big\{\lambda=\rho \,e^{-i\phi} : \rho_0\le \rho<\infty\big\}\cup
\big\{\lambda=\rho_0\,e^{i\theta} : |\theta|\le\phi\big\}\cup
\big\{\lambda=\rho\,e^{i\phi} : \rho_0\le\rho<\infty\big\}
\end{equation}
\end{subequations}
is oriented so that its orientation is counter-clock-wise around the origin. The semigroup  satisfies 
$\|e^{z\Lap}\|_{\mathcal{L}(X)}\le M_{\epsilon,\delta}\,e^{\omega|z|}$ for $z\in {\bbS}_{\vartheta-\epsilon}$
and any $\epsilon\in(0,\vartheta)$.

\begin{Prop}\label{prop:e^tDelta-H} 
For $1< p<\infty$ and for any $m\in\Z_{\ge 0}$ and $\delta\in\R$  the operator $\Delta$ with domain 
$D=H^{m+2,p}_\delta(\R^d)$ generates an analytic semigroup $e^{z\Delta}$ of angle $\pi/2$ on 
$H^{m,p}_\delta(\R^d)$ and for any given $\omega>0$ and $\epsilon\in(0,\pi/2)$ it satisfies \eqref{eq:Inf-Gen-Est}
with $\vartheta=\pi/2$ and $X\equiv H^{m,p}_\delta(\R^d)$.  
\end{Prop}

\begin{Rem}
Proposition \ref{prop:e^tDelta-H} with a weaker version of the inequality \eqref{eq:Inf-Gen-Est} can be deduced 
from \cite[Theorem 9.4]{AmannHieberSimonett} in the case when $m=0$.
\end{Rem}

\begin{proof}[Proof of Proposition \ref{prop:e^tDelta-H}]
By Corollary \ref{Co:apriori-estimates} in Appendix B, we know that the unbounded operator $\Lap$ on 
$H^{m,p}_\delta(\R^d)$ with domain $D=H^{m+2,p}_\delta(\R^d)$ is a closed operator, and it is densely defined 
(since $C^\infty_0(\R^d)$ is dense in both $H^{m,p}_\delta(\R^d)$ and $H^{m+2,p}_\delta(\R^d)$); so we need 
only prove the resolvent estimate \eqref{eq:Inf-Gen-Est}.
We start with $m=0$. 
Using Lemma \ref{le:J_delta-H}, the resolvent $R(\lambda)=(\lambda-\Lap)^{-1}$ on $L^{p}_\delta(\R^d)$ is 
equivalent to $J_{\delta}(\lambda-\Lap)^{-1}J_{-\delta}$ on  $L^{p}(\R^d)$. Let us first consider estimates for $d=3$, 
since then we have a simple formula for the integral kernel of $(\lambda-\Lap)^{-1}$ (cf.\ \cite{Kato}, p.\ 493):
\begin{equation}\label{eq:K(x,y,lambda),d=3}
K(x,y;\lambda)=\frac{1}{4\pi|x-y|}\exp\left(-\sqrt{\lambda}\,|x-y|\right) \quad\hbox{for}\ \re\sqrt{\lambda}>0.
\end{equation}
The integral kernel for $J_{\delta}(\lambda-\Lap)^{-1}J_{-\delta}$ on $L^{p}(\R^d)$ is
\begin{equation}\label{eq:K_delta(x,y,lambda),d=3}
K_{\delta}(x,y;\lambda)=\frac{\x^\delta\y^{-\delta}}{4\pi|x-y|} 
\exp\left(-\sqrt{\lambda}\,|x-y|\right) \quad\hbox{for}\ \re\,\sqrt{\lambda} >0.
\end{equation}
By the Schur test for $L^p$-boundedness, we want to estimate $\int_{\R^3} |K_{\delta}(x,y;\lambda)|\,dx$ for 
$y\in \R^3$ and $\int_{\R^3} |K_{\delta}(x,y;\lambda)|\,dy$ for $x\in \R^3$.
Using $\x^\delta\y^{-\delta}\le C\langle x-y\rangle^{|\delta|}$ (cf.\ Lemma \ref{Le:elementary}) and a change of 
variables,  it suffices to estimate
\[
I:=\frac{1}{4\pi}\int_{\R^3} \frac{\x^{|\delta|}}{|x|} e^{-\re\,\sqrt{\lambda}\,|x|}\,dx
=\int_0^\infty (1+r^2)^{|\delta|/2}\,r\, e^{-\re\,\sqrt{\lambda}\,r}\,dr.
\]
After another change of variables, we have
 \[
I=\frac{1}{\rm (Re\sqrt{\lambda})^2 }\int_0^\infty 
\left(1+\frac{s^2}{({\rm Re\sqrt{\lambda}})^2}\right)^{|\delta|/2}s\,e^{-s}\,ds.
\]
If $\delta=0$, we use $\int_0^\infty s\,e^{-s}\,ds=1$ to obtain the $L^p$-boundedness
 \[
\big\|R(\lambda)\big\|_{\mathcal{L}(L^p)}
\leq\frac{1}{\rm (Re\sqrt{\lambda})^2 }\le \left(|\lambda|\sin^2\frac{\epsilon}{2}\right)^{-1}\quad\hbox{for}\  
|{\rm arg}\,\lambda|\le \pi-\epsilon,
\]
which is the estimate \eqref{eq:Inf-Gen-Est} with $\vartheta=\pi/2$ and $\omega=0$, showing the well-known result that 
$\Lap$ generates a bounded analytic semgroup of angle $\pi/2$ on $L^p(\R^3)$. But, of course, we are interested in 
$\delta\not=0$, where $\int_0^\infty \left(1+s^2/({\rm Re\sqrt{\lambda}})^2\right)^{|\delta|/2}s\,e^{-s}\,ds$ becomes 
infinite as $\lambda\to 0$. However, if we restrict $\lambda$ to lie outside the ball 
$B_\kappa(0):=\{\lambda\in\C : |\lambda|\le\kappa\}$ with $\kappa>0$, then we get the following estimate:
\begin{equation}\label{eq:resolventestimate2}
\big\|R(\lambda)\big\|_{\mathcal{L}(L^p_\delta)}\le \frac{C_{\delta,\kappa,\epsilon}}{|\lambda| }\quad\hbox{for}\ \lambda\in 
\Sigma_{\pi/2-\epsilon,0}\backslash B_\kappa(0).
\end{equation}
Finally, for any $\omega>0$ and $\epsilon\in(0,\pi/2)$ we can choose $\kappa$ so that $\Sigma_{\pi/2-\epsilon,\omega}\subseteq 
\Sigma_{\pi/2-\epsilon,0}\backslash B_\kappa(0)$, 
and then use $|\lambda-\omega|\le C_{\epsilon,\omega}|\lambda|$ for $\lambda\in\Sigma_{\pi/2-\epsilon,\omega}$ to conclude
\begin{equation}\label{eq:resolventestimate3}
\big\|R(\lambda)\big\|_{\mathcal{L}(L^p_\delta)}\le\frac{C_{\delta,\epsilon,\omega}}{|\lambda-\omega| }\quad\hbox{for}\ 
\lambda\in\Sigma_{\pi/2-\epsilon,\omega},
\end{equation}
which is the estimate \eqref{eq:Inf-Gen-Est} showing that $\Lap$ generates an analytic semgroup of angle $\pi/2$ on 
$L_\delta^p(\R^3)$.

For general $d\ge 2$, the integral kernel of $(\lambda-\Lap)^{-1}$ for $\re\sqrt{\lambda}>0$ is given by (cf.\ \cite{Taylor} 
and  \cite{Lebedev})
\begin{equation}\label{eq:K(x,y,lambda)}
K(x,y;\lambda)=\frac{i}{4} \left(\frac{i\sqrt{\lambda}}{2\pi|x-y|}\right)^\nu H_\nu^{(1)}\left(i\sqrt{\lambda}\,|x-y|\right), 
\quad \nu=(d-2)/2,
\end{equation}
where $H_{\nu}^{(1)}(z)$ is the first Hankel function. (For $d=3$, $\nu=1/2$ and 
$H^{(1)}_{1/2}(z)=-i\left(2/\pi z\right)^{1/2}e^{iz}$, so \eqref{eq:K(x,y,lambda)} reduces to 
\eqref{eq:K(x,y,lambda),d=3}.) To obtain an estimate for $(\lambda-\Lap)^{-1}$ on $L^p_\delta(\R^d)$, we need to 
estimate the operator $J_\delta (\lambda-\Lap)^{-1}J_{-\delta}$ on $L^p(\R^d)$ with integral kernel
\begin{equation}\label{eq:K_delta(x,y,lambda)}
K_\delta(x,y;\lambda)=\frac{i}{4} \left(\frac{i\sqrt{\lambda}}{2\pi}\right)^\nu \frac{\x^\delta\y^{-\delta}}{|x-y|^\nu} 
\,H_\nu^{(1)}\left(i\sqrt{\lambda}\,|x-y|\right).
\end{equation}
In view of Lemma \ref{Le:elementary}, to estimate $\int_{\R^d} |K_{\delta}(x,y;\lambda)|\,dx$ for $y\in \R^d$ and 
$\int_{\R^d} |K_{\delta}(x,y;\lambda)|\,dy$ for $x\in \R^d$ it is enough to estimate
\[
I:=|\lambda|^{\nu/2}\int_0^\infty (1+r^2)^{|\delta|/2}\,r^{d-1-\nu}\left| H_\nu^{(1)}
\left(i\sqrt{\lambda}\,r\right) \right| \,dr\,.
\]
We split this into two integrals
\[
\begin{aligned}
I_1:&= |\lambda|^{\nu/2}\int_0^{(\re\sqrt{\lambda})^{-1}}(1+r^2)^{|\delta|/2}\,r^{d-1-\nu}
\left| H_\nu^{(1)}\left(i\sqrt{\lambda}\,r\right) \right| \,dr,
\\
I_2:&= |\lambda|^{\nu/2}\int_{(\re\sqrt{\lambda})^{-1}}^\infty(1+r^2)^{|\delta|/2}\,r^{d-1-\nu}
\left| H_\nu^{(1)}\left(i\sqrt{\lambda}\,r\right) \right| \,dr.
\end{aligned}
\]
Recall (cf.\ \cite{Lebedev}) the asymptotic behaviors
\begin{equation}\label{eq:HankelAsym}
H_\nu^{(1)}(z) = \begin{cases}
-\frac{i}{\pi} \Gamma(\nu)\left(\frac{2}{z}\right)^\nu+O(|z|^{2-\nu})\quad\hbox{as $|z|\to\ 0$, $|{\rm arg}(z)|<\pi$}, \\
\sqrt{\frac{2}{\pi z}} \,\exp\Big(i\left(z-\frac{\nu\pi}{2}-\frac{\pi}{4}\right)\Big)
\left(1+O(|z|^{-1})\right)\quad\hbox{as $|z|\to\infty$, 
$|{\rm arg}(z)|<\pi-\varepsilon$}.
\end{cases}
\end{equation}
Now $\re\sqrt{\lambda}=|\lambda|^{1/2}\cos\big(\frac{1}{2}\arg\lambda\big)$ implies 
\[
0<\re\sqrt\lambda\le |\lambda|^{1/2}\le C\,\re\sqrt{\lambda}\quad\hbox{for}\ \lambda\in  \Sigma_{\pi/2-\epsilon,0},
\]
with a constant $C>0$ depending on $\epsilon>0$.
So $0<r<(\re\sqrt{\lambda})^{-1}$ implies that $z=i\sqrt{\lambda}\,r$ satisfies $|z|\le C$ and we can
use the first line in \eqref{eq:HankelAsym} to estimate 
$|\lambda|^{\nu/2}\big|H_\nu^{(1)}(i\sqrt\lambda \,r)\big|\le C_\epsilon r^{-\nu}$
uniformly in $\lambda\in\Sigma_{\pi/2-\epsilon,0}$ and $0<r<(\re\sqrt{\lambda})^{-1}$.
Consequently, using $|\lambda|>\kappa$ and $d-1-2\nu=1$, we obtain
\[
I_1\le C_{\epsilon}\,\int_0^{(\re\sqrt{\lambda})^{-1}}\big(1+r^2\big)^{|\delta|/2}\,r^{d-1-2\nu}\,dr
\le C_{\delta,\kappa,\epsilon}'\int_0^{(\re\sqrt{\lambda})^{-1}} r\,dr \le 
\frac{C_{\delta,\kappa,\epsilon}'}{2(\re\sqrt{\lambda})^{2}} \le \frac{C_{\delta,\kappa,\epsilon}}{|\lambda|}
\]
for some constants $C_{\delta,\kappa,\epsilon}'>0$ and $C_{\delta,\kappa,\epsilon}>0$.
Similarly, for $r>(\re\sqrt{\lambda})^{-1}$ we have that $z=i\sqrt{\lambda}\,r$ satisfies $|z|\ge 1$ so we can use 
the second line in \eqref{eq:HankelAsym} to estimate 
$\big|H^{(1)}_\nu(i\sqrt{\lambda}\,r)\big|\le C|\lambda|^{-1/4}r^{-1/2}e^{-r\re(\sqrt{\lambda})}$. Consequently, 
using $\frac{\nu}{2}-\frac{1}{4}=\frac{d-3}{4}$ and $d-1-\nu-\frac{1}{2}=\frac{d-1}{2}$, we can estimate
\[
\begin{aligned}
I_2 & \le C\, |\lambda|^{\frac{\nu}{2}-\frac{1}{4}} \int_{(\re\sqrt{\lambda})^{-1}}^\infty 
(1+r^2)^{|\delta|/2}\,r^{d-1-\nu-\frac{1}{2}}e^{-\re\sqrt{\lambda}\,r} \,dr \\
& = C\,|\lambda|^{\frac{d-3}{4}} \int_{(\re\sqrt{\lambda})^{-1}}^\infty 
(1+r^2)^{|\delta|/2}\,r^{\frac{d-1}{2}}e^{-\re\sqrt{\lambda}\,r} \,dr
\end{aligned}
\]
for some constant $C>0$. Now we change integration variable to $s=(\re\sqrt{\lambda})\,r$ to obtain that for $|\lambda|>\kappa$
and $\lambda\in\Sigma_{\pi/2-\epsilon,0}$,
\[
\begin{aligned}
I_2 &\le C\,|\lambda|^{\frac{d-3}{4}} \int_1^\infty \left(1+\frac{s^2}{(\re\sqrt{\lambda})^{2}}
\right)^{|\delta|/2}\frac{s^{(d-1)/2}}{(\re\sqrt{\lambda})^{(d-1)/2}}\,e^{-s}\,\frac{ds}{\rm Re\sqrt{\lambda}} \\
&\le \frac{C_\epsilon}{|\lambda|} \int_1^\infty 
\left(1+\frac{s^2}{(\re\sqrt{\lambda})^{2}}\right)^{|\delta|/2} s^{(d-1)/2} e^{-s}\,ds
=\frac{C_{\delta,\kappa,\epsilon}}{|\lambda|}.
\end{aligned}
\]
This establishes \eqref{eq:resolventestimate2}, and hence \eqref{eq:resolventestimate3}, for  general $d\ge 2$.

Notice that \eqref{eq:resolventestimate3} implies
\begin{equation}\label{eq:resolventestimate4}
\|f\|_{L^p_\delta} \le \frac{C_{\delta,\epsilon,\omega}}{|\lambda-\omega| } \|(\lambda-\Delta)f\|_{L_\delta^p} \quad
\hbox{for}\ \lambda\in 
\Sigma_{\pi/2-\epsilon,\omega}\ \hbox{and}\ f\in H^{2,p}_\delta(\R^d).
\end{equation}
Conversely, if $g\in L_\delta^p(\R^d) $ then $f=(\lambda-\Lap)^{-1}g$ satisfies
$(\lambda-\Lap)f\in L_\delta^p(\R^d) $, and we know by elliptic regularity\footnote{See Lemma \ref{Le:apriori-est-1} in 
Appendix B.} that $f\in H^{2,p}_\delta(\R^d)$. Hence we can apply \eqref{eq:resolventestimate4} to obtain  
\eqref{eq:resolventestimate3}.

Now let us consider $m\ge 1$. 
For $f\in H^{m+2,p}_\delta(\R^d)$ and $|\alpha|\le m$ we have $\partial^\alpha f\in H^{2,p}_\delta(\R^d)$, so we can
apply \eqref{eq:resolventestimate4} to conclude
\[
\|\partial^{\alpha}f\|_{L^{p}_\delta} \le \frac{C_{\delta,\epsilon,\omega}}{|\lambda-\omega|}\,
\|(\lambda I-\Lap)\partial^\alpha f\|_{L^{p}_\delta}
=\frac{C_{\delta,\epsilon,\omega}}{|\lambda-\omega|}\,\|\partial^\alpha\left((\lambda I-\Lap) f\right)\|_{L^{p}_\delta}.
\]
Summing this over $|\alpha|\le m$, we obtain 
\begin{equation}\label{eq:Inf-Gen-Est-4}
\|f\|_{H^{m,p}_\delta}\le\frac{C_{\delta,\epsilon,\omega}}{|\lambda-\omega|}\,\|(\lambda I-\Lap)f\|_{H^{m,p}_\delta}
\quad\hbox{for all $f\in H^{m+2,p}_\delta(\R^d)$ and $\lambda\in\Sigma_{\pi/2-\epsilon,\omega}$}.
\end{equation}
As observed above, this and Corollary \ref{Co:apriori-estimates} in Appendix \ref{sec:aux-lemmas} imply the desired resolvent estimate, 
so $\Lap$ with domain $H^{m+2,p}_\delta(\R^d)$ generates an analytic semigroup of angle $\pi/2$ on $H_\delta^{m,p}(\R^d)$ 
(see e.g. \cite[\S\,2.5, Theorem 5.2]{Pazy}).
\end{proof}

\medskip
In order to prove Theorem \ref{th:SonH}, it remains to show $S(t)=e^{t\Lap}$ on $H^{m,p}_\delta(\R^d)$. For this, 
we need a uniqueness result that we shall prove in a more general context. Let ${\mathcal S}'(\R^d)$ denote the space of 
tempered distributions, i.e.\ the dual of the Schwartz class ${\mathcal S}(\R^d)$. For any $v\in{\mathcal S}'(\R^d)$ we
define $\Delta v\in{\mathcal S}'(\R^d)$ using distributional derivatives: 
$\langle \Delta v,\phi\rangle = \langle  v,\Delta\phi\rangle$ for $\phi\in {\mathcal S}(\R^d)$. 
We consider solutions of \eqref{eq:heat} when $v\in {\mathcal S}'(\R^d)$. 
Recall that $w\in C^k((0,\infty),{\mathcal S}' (\R^d))$ for a given integer $k\ge 0$ if 
for any test function $\phi\in{\mathcal S}(\R^d)$ the map $t\mapsto\langle w(t),\phi\rangle$, $(0,\infty)\to\R$, is $k$-times
continuously differentiable. The following uniqueness result is well-known. 
The proof is provided for the convenience of the reader.

\begin{Lem}\label{le:uniqueness} 
Suppose that $w\in C([0,\infty),{\mathcal S}'(\R^d))\cap C^1((0,\infty),{\mathcal S}'(\R^d))$ satisfies $w_t(t)=\Delta w(t)$ 
as distributions for $t>0$ and $w|_{t=0}=0$. Then $w(t)=0$ for all $t\ge 0$.
\end{Lem}

\proof
For $t\ge 0$ we have $w(t)\in{\mathcal S}'(\R^d))$ and its Fourier transform $\widehat{w}(t)\in {\mathcal S}'(\R^d)$
is defined by $\langle \widehat{w}(t),\phi\rangle=\langle w(t),\widehat{\phi}\rangle$ for all $\phi\in{\mathcal S}(\R^d)$. 
But $t\mapsto w(t)$ is $C^1$ as a map $(0,\infty)\to {\mathcal S}'(\R^d)$, so we can compute
\[
\partial_t\big\langle\widehat w(t),\phi\big\rangle=
\partial_t\big\langle  w(t),\widehat\phi\,\big\rangle=\big\langle\partial_t w(t),\widehat\phi\,\big\rangle=
\big\langle\widehat{\partial_t w}(t),\phi\big\rangle\quad\hbox{for all $\phi\in{\mathcal S}(\R^d)$}
\]
to conclude $\partial_t\widehat w=\widehat{\partial_t w}$ for $t>0$. This enables us to take the Fourier transform 
of $w_t=\Lap w$ to conclude that $\widehat{w}\in{\mathcal S}'(\R^d)$ satisfies $\widehat w_t(\xi,t)=-|\xi|^2\widehat w(\xi,t)$.
This implies that for any test function with compact support $\psi\in C^\infty_c(\R^d)$ and for any $t>0$ we have that
\[
\frac{d}{dt}\big\langle e^{t|\xi|^2}\widehat{w}(\xi,t),\psi(\xi)\big\rangle=0
\]
where $\frac{d}{dt}$ denotes the (pointwise) derivative in $t$.
Hence, $e^{t|\xi|^2}\widehat{w}(\xi,t)$ considered as a distribution in ${\mathcal D}'(\R^d)$
is independent of $t>0$. But since $\widehat{w}\in C([0,\infty],{\mathcal S}'(\R^d))$ we then conclude that
$t\mapsto e^{t|\xi|^2}\widehat{w}(\xi,t)$, $[0,\infty)\to{\mathcal D}'(\R^d)$, is continuous at $t=0$.
Since $\widehat{w}|_{t=0}=0$ we then see that $\widehat{w}(t)=0$ for any $t\ge 0$. 
By applying the inverse Fourier transform in ${\mathcal S}'(\R^d)$ to $\widehat{w}(t)$ we then complete 
the proof of the lemma.
\endproof
  
\begin{proof}[Proof of Theorem \ref{th:SonH}]
First, note that \eqref{eq:Gaussian} extends by duality to $\Sz'(\R^d)$ so that for any $g\in\Sz'(\R^d)$ and 
for any $t\ge 0$, $v(t)\equiv S(t)g$ satisfies 
\[
v\in C([0,\infty),\Sz'(\R^d))\cap C^1((0,\infty),\Sz'(\R^d))
\]
and $v_t=\Delta u$, $v|_{t=0}=g$, in distributional sense (cf. e.g. \cite[Ch. 7]{Shubin}).
Now, take $g\in H^{m,p}_\delta$, $m\in\Z_{\ge 0}$.
From the analytic semigroup theory we know that 
$u=e^{t\Lap}g\in C([0,\infty),H^{m,p}_\delta(\R^d))\cap C^1((0,\infty),H^{m,p}_\delta(\R^d))$ and 
$u_t=\Delta u$, $u|_{t=0}=g$. Since 
\[
C([0,\infty),H^{m,p}_\delta(\R^d))\cap C^1((0,\infty),H^{m,p}_\delta(\R^d))\subseteq
C([0,\infty),{\mathcal S}'(\R^d))\cap C^1((0,\infty),{\mathcal S}'(\R^d))
\] 
we see that $w:=u-v$ satisfies the conditions of Lemma \ref{le:uniqueness} and hence $w\equiv 0$. 
This proves that $S(t)=e^{t\Lap}$ on $H^{m,p}_\delta(\R^d)$. 
Theorem \ref{th:SonH} then follows from Proposition \ref{prop:e^tDelta-H} and 
Proposition \ref{prop:S-estimates} below.
\end{proof}

Finally, we consider the effect of spatial derivatives on $S(t)$. These results are known for $L^p(\R^d)$ and 
$H^{m,p}(\R^d)$, but for the weighted Sobolev spaces $H^{m,p}_\delta(\R^d)$, one encounters the growth factor 
$(1+t)^{|\delta|/2}$, so we give the proofs.

\begin{Prop}\label{prop:S-estimates}
For $1\le p<\infty$, $m\in\Z_{\ge 0}$, $\delta\in\R$, and for any multi-index $|\alpha|\ge 0$ the map 
$z\mapsto\partial^\alpha S(z)$, $S_{\pi/2}\to\mathcal{L}\big(H^{m,p}_\delta\big)$, is well-defined and 
for any $\epsilon\in(0,\pi/2)$ there exists a positive constant $C\equiv C_{p,\delta,\epsilon}>0$ such that for any $m\in\Z_{\ge 0}$,
\begin{equation}\label{eq:DS-estimates}
\big\|\partial^\alpha S(z)\big\|_{\mathcal{L}(H^{m,p}_\delta)}\le\frac{C}{|z|^{|\alpha|/2}}\,\big(1+|z|\big)^{|\delta|/2}
\end{equation}
for all $z\in\bbS_{\pi/2-\epsilon}\setminus\{0\}$.
\end{Prop}

\begin{proof}
Since the proof of the proposition is similar for any choice $|\alpha|\ge 0$ we will assume 
without loss of generality that $\partial^\alpha=\partial_j$ for some $1\le j\le d$.
Let us first consider the case when $m=0$. Note that for any $1\le j\le d$ the operator 
$\partial_j S(t) : L^p_\delta(\R^d)\to L^p_\delta(\R^d)$ is equivalent to an integral operator 
on $L^p(\R^d)$ with kernel
\[
K_\delta^j(x,y;t)\equiv\x^\delta G_t^j(x-y)\y^{-\delta}\quad\text{where}\quad
G_t^j(x)\equiv-\frac{x_j}{2t(4\pi t)^{d/2}} e^{-|x|^2/4 t}\,.
\]
This formula is well defined for $t$ replaced by $z\in\bbS_{\pi/2}\setminus\{0\}$.
Moreover, for $z\in\bbS_{\pi/2-\epsilon}\setminus\{0\}$ we have from Lemma \ref{Le:elementary} that
\[
\big|K_\delta^j(x,y;z)\big|\le C\,\langle x-y\rangle^{|\delta|}\,|z|^{-d/2}\,\frac{|x-y|}{|z|}\, 
e^{-\frac{|x-y|^2}{4|z|}\sin\epsilon},\quad x,y\in\R^d.
\]
for some constant $C>0$ depending on $\delta>0$. Then, by Young's inequality, 
\[
\|\partial_j S(z)f\|_{L^p}\le\big\|G_{z,\delta,\epsilon}\ast|f|\big\|_{L^p}\le\|G_{z,\delta,\epsilon}\|_{L^1}\|f\|_{L^p}
\]
where
\[
G_{z,\delta,\epsilon}(x):= C\,\x^{|\delta|}\,|z|^{-d/2}\,\frac{|x|}{|z|}\, e^{-\frac{|x|^2}{4|z|}\sin\epsilon},
\quad x\in\R^d,
\]
with a constant $C>0$ depending on $\delta>0$ and $1\le p<\infty$.
This implies that the operator norm of $\partial_j S(z) : L^p_\delta(\R^d)\to L^p_\delta(\R^d)$ is bounded above by 
the $L^1$-norm of $G_{z,\delta,\epsilon}$. By a change of variable to $y=|z|^{-1/2}\,x$ in the integral below, we obtain
\[
\|G_{z,\delta,\epsilon}\|_{L^1}=C\,\int_{\R^d}\x^{|\delta|}\,|z|^{-d/2}\,\frac{|x|}{|z|}\, 
e^{-\frac{|x|^2}{4|z|}\sin\epsilon}\,dx\le\frac{C_{p,\delta,\epsilon}}{\sqrt{|z|}}\,(1+|z|)^{|\delta|/2}
\]
with a constant $C_{p,\delta,\epsilon}>0$ depending on $1\le p<\infty$, $\delta\in\R$ and $\epsilon\in(0,\pi/2)$, and
such that $C_{p,\delta,\epsilon}\to\infty$ as $\epsilon\to 0$. This proves inequality \eqref{eq:DS-estimates} for $m=0$. 

Let us now consider the case when $m\ge 1$. It follows easily by approximation that for any choice of
the multi-indexes $|\alpha|\ge 0$ and $0\le|\beta|\le m$, and for any $f\in H^{m,p}_N(\R^d)$ we have that
$\partial^\beta\big(\partial^\alpha S(z) f\big)=\big(\partial^\alpha S(z)\big)(\partial^\beta f)\in L^p_\delta(\R^d)$
for any $z\in\bbS_{\pi/2}\setminus\{0\}$. Then, by applying \eqref{eq:DS-estimates} in the case when $m=0$, 
we obtain that for any $0\le|\beta|\le m$,
\begin{equation}\label{eq:DS-estimate(m=0)}
\big\|\partial^\beta\big(\partial^\alpha S(z) f\big)\big\|_{L^p_\delta}=
\big\|\big(\partial^\alpha S(z)\big)(\partial^\beta f)\big\|_{L^p_\delta}\le
\frac{C}{|z|^{|\alpha|/2}}\,\big(1+|z|\big)^{|\delta|/2}\big\|\partial^\beta f\big\|_{L^p_\delta},
\end{equation}
for any given $|\alpha|\ge 0$, $f\in H^{m,p}_N(\R^d)$, and $z\in\bbS_{\pi/2}\setminus\{0\}$.
By summing this over $0\le|\beta|\le m$, we obtain \eqref{eq:DS-estimates} with the same constant as
in \eqref{eq:DS-estimate(m=0)}. In particular, the constant in \eqref{eq:DS-estimates} does not depend on $m\ge 0$.
\end{proof}

\begin{Coro}\label{coro:DS-estimate} 
For $1< p<\infty$, $m\in\Z_{\ge 0}$, $\delta\in\R$, the map $z\mapsto S(z)$, 
$S_{\pi/2}\setminus\{0\}\to\mathcal{L}\big(H^{m,p}_\delta,H^{m+1,p}_\delta\big)$, 
is analytic and for any $\epsilon\in(0,\pi/2)$ there exists a positive constant 
$C\equiv C_{p,\delta,\epsilon}>0$ such that for any $m\in\Z_{\ge 0}$,
\begin{equation}\label{eq:DS-estimate}
\big\|S(z)\big\|_{\mathcal{L}(H^{m,p}_\delta,H^{m+1,p}_\delta)}\le
C\,\max\big({|z|}^{-1/2},1\big)\,\big(1+|z|\big)^{|\delta|/2}
\end{equation}
for all $z\in\bbS_{\pi/2-\epsilon}\setminus\{0\}$.
\end{Coro}

\proof
Take $\epsilon\in(0,\pi/2)$. Then, by Propositions \ref{prop:S-estimates}, for any $z\in\bbS_{\pi/2-\epsilon}\setminus\{0\}$
and for any $f\in H^{m,p}_\delta(\R^d)$ we have that $S(z)f\in H^{m+1,p}_\delta$ and for some constant $C>0$,
\[
\begin{aligned}
\|S(z)f\|_{H^{m+1,p}_\delta}  &\le\|\nabla S(z)f\|_{H^{m,p}_\delta} + \|S(z)f\|_{H^{m,p}_{\delta}} \\
&\le C\,|z|^{-1/2} (1+|z|)^{|\delta|/2}\,\|f\|_{H^{m,p}_{\delta}} + C\,(1+|z|)^{|\delta|/2}\,\|f\|_{H^{m,p}_{\delta}},
\end{aligned}
\]
from which \eqref{eq:DS-estimate} easily follows.

Now we prove the analyticity of the map 
$z\mapsto S(z)$, $S_{\pi/2}\setminus\{0\}\to\mathcal{L}\big(H^{m,p}_\delta,H^{m+1,p}_\delta\big)$. 
Let $K\subseteq\C$ be a closed disk in $\bbS_{\pi/2}\setminus\{0\}$ and take an arbitrary $f\in H^{m,p}_\delta(\R^d)$. 
We first claim that $z\mapsto S(z)f$ is an analytic map $\intr{K}\to H^{m+1,p}_\delta(\R^d)$, where $\intr{K}$ denotes 
the interior of $K$. In fact,  let $(f_k)_{k\ge 1}$ be a sequence in $C^\infty_c(\R^d)\subseteq H^{m+1,p}_\delta(\R^d)$ such 
that $f_k\to f$ in $H^{m,p}_\delta(\R^d)$ as $k\to\infty$. By  \eqref{eq:DS-estimate}, there exists $C>0$ such that 
for any $z\in K$ we have $\|S(z)\|_{\mathcal{L}(H^{m,p}_\delta,H^{m+1,p}_\delta)}\le C$ which implies that
\[
\big\|S(z)f-S(z)f_k\big\|_{H^{m+1,p}_\delta}\le C\,\|f-f_k\|_{H^{m,p}_\delta}
\]
uniformly in $z\in K$ and $k\ge 1$. Hence, $S(z)f_k\to S(z)f$ in $H^{m+1,p}_\delta(\R^d)$ as $k\to\infty$, 
uniformly in $z\in K$. On the other hand, from Proposition \ref{prop:e^tDelta-H} and the fact that for each $k\ge 1$,
$f_k\in C^\infty_c(\R^d)\subseteq H^{m+1,p}_\delta(\R^d)$ we obtain that $z\mapsto S(z)f_k$, 
$\intr{K}\to H^{m+1,p}_\delta(\R^d)$, is analytic. Using that the uniform limit of analytic maps is analytic we then
conclude that $z\mapsto S(z)f$, $\intr{K}\to H^{m+1,p}_\delta(\R^d)$, is analytic. 
Since for any $z\in S_{\pi/2}\setminus\{0\}$, $S(z)\in\mathcal{L}\big(H^{m,p}_\delta,H^{m+1,p}_\delta\big)$,
we then obtain from \cite[Theorem 3.12 in Chapter 3]{Kato} that $z\mapsto S(z)$, 
$S_{\pi/2}\setminus\{0\}\to\mathcal{L}\big(H^{m,p}_\delta,H^{m+1,p}_\delta\big)$, is analytic. 
\endproof

\section{The semigroup on $\AH^{m,p}_N$}\label{sec:asymptotic_spaces}
In this section we assume $1<p<\infty$ and $m\in\N$ with  $m>d/p$; consequently, if $v\in\AH^{m,p}_N$ then 
$v\in CB(\R^d)$ by Proposition \ref{pr:A-properties}(d) and $S(t)v$ may be defined for $t>0$ by \eqref{eq:Gaussian}. 
We will show that $\{S(t)\}_{t\ge 0}$ extends to an analytic semigroup of angle $\pi/2$ on $\AH^{m,p}_N$.

\begin{Th}\label{th:SonAH} 
For $1<p<\infty$ and for any $m, N\in\Z_{\ge 0}$ with $m>d/p$, formula \eqref{eq:Gaussian} defines a strongly continuous 
semigroup $\{S(t)\}_{t\ge 0}$ on $\AH_N^{m,p}$ satisfying \eqref{est:S(t)onAH}. 
The generator is an operator $\Lambda$ with domain $D$ containing $\AH^{m+2,p}_N$ on which $\Lambda v=\Lap v$.
Moreover, $\{S(t)\}_{t\ge 0}$ is analytic of angle $\pi/2$ such that for any $\epsilon\in(0,\pi/2)$ there exists 
a positive constant $C\equiv C_{p,\delta,\epsilon,N}>0$ such that for any $m\in\Z_{\ge 0}$,
\begin{equation}\label{eq:S-estimate on AH}
\|S(z)\|_{\mathcal{L}(\AH^{m,p}_N)}\le C\,\big(1+|z|\big)^{\mu},\quad\mu\equiv\big(N+N^*+2\big)/2,
\end{equation}
for all $z\in\bbS_{\pi/2-\epsilon}\setminus\{0\}$.
Finally, if $N\ge 2$ and $v\in\AH_N^{m,p}$ then $S(z)v-v\in\AH_{2,N}^{m,p}$ for all $z\in {\bbS}_{\pi/2}$.
\end{Th}

\begin{Rem}\label{Rem:NotClosed} 
Comparing with Theorem \ref{th:SonH}, one might expect that the generator of $S(t)$ on $\AH_N^{m,p}$ is $\Lap$
with domain $D=\AH_N^{m+2,p}$. However, the unbounded operator  $\Lap$ with domain $D=\AH_N^{m+2,p}$ is  
not closed on $\AH_N^{m,p}$, so cannot be the generator of $S(t)$. This is further discussed in Appendix \ref{sec:aux-lemmas}.
\end{Rem}

\begin{Rem}\label{Rem:n>0} 
For $0<n\le N$, we can consider $S(z)$ as an analytic semigroup of angle $\pi/2$ on $\AH_{n,N}^{m,p}$. All results in this 
section apply to this more general setting.
\end{Rem}

\begin{Rem}\label{Rem:S(z)v-v interp} 
The statement $S(z)v-v\in \AH_{2,N}^{m,p}$ means that the asymptotic coefficients $a_0(z)$ and $a_1(z)$ of
the solution $u(z)\equiv S(z)v$ remain invariant under the heat flow. Similarly, if we consider $S(z)$ as an analytic semigroup on 
$\AH_{n,N}^{m,p}$ and $N\ge n+2$, then $S(z)v-v\in \AH^{m,p}_{n+2,N}$, so the coefficients
$a_n(z)$ and $a_{n+1}(z)$ of the solution remain invariant.
\end{Rem}

For simplicity, we first tackle part of the proof of Theorem \ref{th:SonAH} in the following

\begin{Prop}\label{pr:SonAH} 
For $1<p<\infty$ and for any $m, N\in\Z_{\ge 0}$ with $m>d/p$, $\{S(t)\}_{t\ge 0}$ is a strongly continuous semigroup on 
$\AH_N^{m,p}$. Moreover, the generator is an operator $\Lambda$  with domain $D$ that contains $\AH_N^{m+2,p}$ and 
$\Lambda v=\Delta v$ for $v\in \AH_N^{m+2,p}$.
\end{Prop}

\begin{proof}[Proof of Proposition \ref{pr:SonAH}]
Let $v\in \AH_N^{m,p}$ and consider $w(t):=S(t)v$. Since $\AH_N^{m,p}$ is contained in $\Sz'(\R^d)$ we have that
$w\in C([0,\infty),\Sz'(\R^d))\cap C^1((0,\infty),\Sz'(\R^d))$ and $w_t=\Delta w$, $w|_{t=0}=v$ (see e.g. \cite{Shubin}). 
On the other hand, by an ansatz, we will construct a solution $u$ of the heat equation $u_t=\Lap u$, $u|_{t=0}=v$; 
as we see below, the solution found this way is in $C([0,\infty),\AH_N^{m,p})\cap C^1((0,\infty),\Sz'(\R^d))$. 
Then, by Lemma \ref{le:uniqueness}, $u(t)=S(t)v$ for $t\ge 0$ and the theorem will follow. 

To give the details, let us write $v\in \AH_N^{m,p}$ as
\begin{equation}\label{def:v}
v(x)=b(x)+g(x)=\chi(r)\left(b_0(\theta)+\cdots+\frac{b_{N^*}(\theta)}{r^{N^*}}\right)+g(x),
\end{equation}
where $b_k\in H^{m+1+N^*-k}(S^{d-1})$, $0\le k\le N^*$, and $g\in H^{m,p}_{N}(\R^d)$.
We seek a solution $u$ of the form
\begin{equation}\label{eq:u-asymptotics}
u(x,t):=a(x,t)+f(x,t)=\chi(r)\left(a_0(\theta,t)+\cdots+\frac{a_{N^*}(\theta,t)}{r^{N^*}}\right)+f(x,t)
\end{equation}
with initial conditions $a_k(\theta,0)=b_k(\theta)$, $0\le k\le N^*$, and $f(x,0)=g(x)$, and such that
\begin{eqnarray*}
&f\in C([0,\infty),L^p_N(\R^d))\cap C^1((0,\infty),L^p_N(\R^d)),\\
&a_k\in C([0,\infty),H^{2,p}(S^{d-1}))\cap C^1((0,\infty),L^p(S^{d-1})),\,\,0\le k\le N^*.
\end{eqnarray*}
If we apply the Laplacian to $u$ in \eqref{eq:u-asymptotics} we obtain
\begin{equation}\label{eq:Lap(u)}
\Lap u=\chi(r)\left(\frac{\Lap_\theta a_0}{r^2}+
\cdots+\frac{\Lap_\theta a_{N^*}+N^*(N^*+2-d)a_{N^*}}{r^{N^*+2}}\right) 
+2\,\nabla\chi\cdot\nabla \tilde a +(\Lap\chi)\tilde a+\Lap f,
\end{equation}
where $\Lap_\theta$ denotes the Laplacian on $S^{d-1}$; here and throughout the rest of this proof we set
\begin{equation}\label{def:tilde-a}
\tilde{a}(t):=a_0(t)+\cdots+\frac{a_{N^*}(t)}{r^{N^*}}
\end{equation}
which implies that $a=\chi\,\tilde{a}$.
We also have
\begin{equation}\label{eq:u_dot}
\dt{u}(t)=\chi(r)\left(\dt{a}_0(t)+\cdots+\frac{\dt{a}_{N^*}(t)}{r^{N^*}}\right)+\dt{f}(t)
\end{equation}
where dot denotes the $t$-derivative.
It now follows from \eqref{def:v}, \eqref{eq:Lap(u)}, and \eqref{eq:u_dot}, that \eqref{eq:u-asymptotics} satisfies the heat equation
$u_t=\Delta u$, $u|_{t=0}=v$, if $a$ and $f$ satisfy the following system of equations:
\begin{subequations}\label{eq:ODEs}
\begin{equation}\label{eq:ODEs-1}
\dt{a}_0(t)=\dt{a}_1(t)=0
\end{equation}
\begin{equation}\label{eq:ODEs-2}
\dt{a}_k(t)=\Lap_\theta a_{k-2}(t) +(k-2)(k-d) a_{k-2}(t)\quad\hbox{for $k=2,\dots,N^*$}
\end{equation}
and
\begin{equation}\label{eq:ODEs-4}
\dt{f}(t)=\Lap f(t)+h(t),
\end{equation}
where
\begin{equation}\label{eq:ODEs-5}
\begin{aligned}
h(t):=&\chi\left(\frac{\Lap_\theta a_{N^*-1}(t)+(N^*-1)(N^*+1-d)\,a_{N^*-1}(t)}{r^{N^*+1}}\right) \\
&+\chi\left(\frac{\Lap_\theta a_{N^*}(t)+N^*(N^*+2-d)\,a_{N^*}(t)}{r^{N^*+2}}\right)+
\,2\nabla\chi\cdot\nabla{\tilde a}(t)+(\Lap\chi){\tilde a}(t),
\end{aligned}
\end{equation}
with initial conditions
\begin{equation}\label{eq:ODEs-6}
a_k(0)=b_k\ \hbox{for $k=0,\dots,N^*$},
\end{equation}
\begin{equation}\label{eq:ODEs-7}
f(0)=g.
\end{equation}
\end{subequations}
From \eqref{eq:ODEs-1} we find that $a_0$ and $a_1$ are independent of $t$:
\begin{equation}\label{eq:ODEs-solution-a}
a_k(t)=b_k\quad\hbox{for $k=0,1$.}
\end{equation}
To find $a_k(t)$ for $k=2,\dots, N^*$, we can iteratively integrate \eqref{eq:ODEs-1}, \eqref{eq:ODEs-2} with initial 
condition \eqref{eq:ODEs-6}, since $a_{k-2}(t)$ has been previously determined. We write this as
\[
a_k(t)=b_k+\int_0^t\Big[\Lap_\theta a_{k-2}(s)+(k-2)(k-d)a_{k-2}(s)\Big]\,ds
\quad\hbox{for $k=2,\dots,N^*$}
\]
which implies
\begin{equation}
\begin{aligned}\label{eq:ak=...}
a_2(t)&=t\,\Lap_\theta b_0+b_2,\\
a_3(t)&=t\left[\Lap_\theta b_1+(3-d)b_1\right]+b_3,\\
a_4(t)&=\frac{t^2}{2}\left[\Lap^2_\theta b_0+2(4-d)\Lap_\theta b_0\right]+t\left[\Lap_\theta b_2+2(4-d)b_2\right]+b_4,\\
&\dots
\end{aligned}
\end{equation}
It follows from \eqref{eq:ak=...} that each $a_k(t)$ is a polynomial in $t$ of degree $\le k/2$ with coefficients in 
$H^{m+1+N^*-k,p}(S^{d-1})$. In particular, $a_k$ belongs to $C^\infty([0,\infty),H^{m+1+N^*-k,p}(S^{d-1}))$.
To find $f$ from \eqref{eq:ODEs-4}, we use Duhamel's formula and set
\begin{equation}\label{eq:f=}
f(t):=S(t)g+\int_0^t S(t-s)h(s)\,ds.
\end{equation}
It follows from  \eqref{def:tilde-a}, \eqref{eq:ODEs-5}, and \eqref{eq:ak=...}, that $h\in C^1([0,\infty),L^p_N(\R^d))$.
Since $g\in H^{m,p}_N(\R^d)$ we then conclude from \eqref{eq:f=}, Theorem \ref{th:SonH}, and \cite[Corollary 3.3, Ch. 4]{Pazy}, that 
\[
f\in C([0,\infty),L^p_N(\R^d))\cap C^1((0,\infty),L^p_N(\R^d))
\]
which implies that $u$ given by \eqref{eq:u-asymptotics} belongs to $C([0,\infty),\Sz'(\R^d))\cap C^1((0,\infty),\Sz'(\R^d))$ and,
by construction, it satisfies $u_t=\Delta u$, $u|_{t=0}=v$. Then, by Lemma \ref{le:uniqueness}, 
\begin{equation}\label{eq:u=Sv}
u(t)\equiv S(t)v\,.
\end{equation}
Our next task is to show that $u\in C([0,\infty),\AH^{m,p}_N)$.
To this end, note that by \eqref{eq:ak=...} we have that for any $t\in [0,\infty)$,
\[
\|a_k(t)\|_{H^{m+1+N^*-k,p}}\le
C\,(1+t)^{[k/2]}\sum_{j=0}^k \|b_j\|_{H^{m+1+N^*-j,p}} 
\]
with a constant $C>0$ independent of $t$ and where $[\cdot]$ denotes the integer part of a real number. 
Summing over $k=0,\dots,N^*$ we obtain
\begin{equation}\label{est:ODEsoln}
\sum_{k=0}^{N^*}\|a_k(t)\|_{H^{m+1+N^*-k,p}}\le C\,(1+t)^{N^*/2}\,\sum_{k=0}^{N^*}\|b_k\|_{H^{m+1+N^*-k,p}}
\end{equation}
with a constant $C>0$ independent of $t$. 

We want to show that $f$ is continuous as a map $[0,\infty)\to H_{N}^{m,p}(\R^d)$ with $f(0)=g$.
Since $g\in H_{N}^{m,p}(\R^d)$, by Proposition \ref{th:SonH} we know that $S(\cdot)g\in C([0,\infty),H_{N}^{m,p}(\R^d))$
and $S(0)g=g$. Now we consider the convolution term 
\[
W(t):=\int_0^t S(t-s)h(s)\,ds
\]
in \eqref{eq:f=}. Note that by \eqref{eq:ODEs-5} and \eqref{eq:ak=...}, $h(t)=h_1(t)+h_2(t)$ where 
\[
\begin{aligned}
&h_1:=\chi\frac{\Lap_\theta a_{N^*}+N^*(N^*+2-d)\,a_{N^*}}{r^{N^*+2}}\in 
C^\infty([0,\infty),H_{N}^{m-1,p}(\R^d)), \ \hbox{and}\\
&h_2:=\chi\,\frac{\Lap_\theta a_{N^*-1}+(N^*-1)(N^*+1-d)\,a_{N^*-1}}{r^{N^*+1}}+
\,2\nabla\chi\cdot\nabla \tilde a+(\Lap \chi)\tilde a
\in C^\infty([0,\infty),H^{m,p}_N(\R^d)).
\end{aligned}
\]
(The  loss of one derivative in $h_1$ comes from $\Lap_\theta a_{N^*}\in H^{m-1,p}(\s^{d-1})$; $\tilde a$ in $h_2$ refers 
to \eqref{def:tilde-a}.)
Using \eqref{est:ODEsoln}, we have 
\[
\begin{aligned}
&\|h_1(t)\|_{H_N^{m-1,p}}\le C\,\|a_{N^*}(t)\|_{H^{m+1,p}}\le 
C\,(1+t)^{N^*/2}\sum_{k=0}^{N^*} \|b_k\|_{H^{m+1+N^*-k,p}},\\
&\|h_2(t)\|_{H_N^{m,p}}\le C\,\sum_{k=0}^{N^*}\|a_{k}(t)\|_{H^{m+1+N^*-k,p}}\le 
C\,(1+t)^{N^*/2}\sum_{k=0}^{N^*} \|b_k\|_{H^{m+1+N^*-k,p}},
\end{aligned}
\]
for some constant $C>0$ independent of $t$. 
By using $\sum_{k=0}^{N^*}\|b_k\|_{H^{m+1+N^*-k,p}}\le \|v\|_{\AH_N^{m,p}}$ we see that
\begin{equation}\label{eq:h-estimate}
\|h(t)\|_{H_N^{m-1,p}}\le C\,(1+t)^{N^*/2}\|v\|_{\AH^{m,p}_N}
\end{equation}
for some constant $C>0$ independent of $t$. From \eqref{eq:DS-estimate} we then obtain that for $0<s<t$,
\begin{equation}\label{eq:S(t)h1}
\begin{aligned}
\|S(s)h_1(t-s)\|_{H^{m,p}_{N}}
\le C\,\max\big(s^{-1/2},1\big)\,(1+s)^{N/2}\,\big(1+(t-s)\big)^{N^*/2}\|v\|_{\AH_N^{m,p}}.
\end{aligned}
\end{equation}
This, together with Corollary \ref{coro:DS-estimate} below, applied with $z\equiv t\in[0,\infty)$, 
then implies that
\[
W(t)=\int_0^t S(s)h(t-s)\,ds\in H_{N}^{m,p}(\R^d),
\] 
$W(t)\to 0$ in $H_{N}^{m,p}(\R^d)$ as $t\to 0$, and $W\in C([0,\infty),H^{m,p}_N(\R^d))$.
In particular, we see from \eqref{eq:f=} that $f\in C([0,\infty),H^{m,p}_{N}(\R^d))$ and $f(0)=g$.
It now follows from \eqref{eq:u-asymptotics} and \eqref{eq:f=} that
\[
u\in C([0,\infty),\AH_N^{m,p}),\quad u(0)=v\,.
\] 
In view of \eqref{eq:u=Sv}, we then obtain that $S(\cdot)v\in C([0,\infty),\AH_N^{m,p})$ for any $v\in \AH_N^{m,p}$ and that
\begin{equation}\label{S(t)v=}
S(t)v=\chi\left(a_0+\frac{a_1}{r}+\frac{a_2(t)}{r^2}+\cdots+\frac{a_{N^*}(t)}{r^{N^*}}\right)
+S(t)g+\int_0^t S(t-s)h(s)\,ds,
\end{equation}
where the $a_k$ and $h$ are given by \eqref{eq:ODEs-solution-a}, \eqref{eq:ak=...}, and \eqref{eq:ODEs-5}.
The semigroup property $S(t_1+t_2)v=S(t_1)S(t_2)v$, $t_1,t_2\ge 0$, is inherited from $\{S(t)\}_{t\ge 0}$ on $\Sz'(\R^d)$. 
Hence, $\{S(t)\}_{t\ge 0}$ is a strongly continuous semigroup on $\AH_N^{m,p}$.

Lastly, we consider the generator $\Lambda$ of the semigroup $\{S(t)\}_{t\ge 0}$,
\begin{equation}\label{def:Lambda}
\Lambda v:=\lim_{t\to 0+0}\frac{S(t)v-v}{t}
\end{equation}
with domain $D(\Lambda):=\big\{v\in\AH^{m,p}_N : \text{the limit \eqref{def:Lambda} exists in}\,\,\AH^{m,p}_N\big\}$. 
If $v=b+g\in\AH_N^{m+2,p}$ as in \eqref{def:v}, we want to show that $v\in D$ and $\Lambda v=\Lap v$.
We can use \eqref{S(t)v=} and \eqref{eq:ak=...} to calculate
\[
\begin{aligned}
\frac{S(t)v-v}{t}=\,&\chi(r)\left(\frac{\Lap_\theta b_0}{r^2}+\frac{\Lap_\theta b_1+(3-d)b_1}{r^3}+
\cdots +\frac{\Lap_\theta b_{N^*-2}+(N^*-2)(N^*-d)b_{N^*-2}}{r^{N^*}}\right) \\
&+O(t)+\frac{S(t)g-g}{t}+\frac{1}{t}\int_0^t S(t-s)h(s)ds,
\end{aligned}
\]
where $O(t)$ denotes a term whose $\AH^{m,p}_N$-norm is bounded by $C\,t\,\|v\|_{\AH^{m,p}_N}$.
By Theorem \ref{th:SonH}, we know that $g\in H^{m+2,p}_{N}(\R^d)$ implies
\[
\lim_{t\to 0+0}\frac{S(t)g-g}{t}=\Lap g\quad\hbox{in}\,\,H^{m,p}_{N}(\R^d).
\]
Moreover, $v\in\AH^{m+2,p}_N$ implies 
$h\in C([0,\infty),H^{m+1,p}_{N}(\R^d))\subseteq C([0,\infty),H^{m,p}_{N}(\R^d))$, so
\begin{equation}\label{eq:lim=h}
\lim_{t\to 0+0}\frac{1}{t}\int_0^t S(t-s)h(s)\,ds= h(0)\quad\hbox{in}\,\,H^{m,p}_{N}(\R^d).
\end{equation}
Consequently, $\lim_{t\to 0}\big(S(t)v-v\big)/t$ in $H^{m,p}_N(\R^d)$ exists, so $v\in D(\Lambda)$ and
\begin{equation}\label{def:w}
\begin{aligned}
\Lambda v=\,\chi(r)\left(\frac{\Lap_\theta b_0}{r^2}+\cdots
+\frac{\Lap_\theta b_{N^*-2}+(N^*-2)(N^*-d)b_{N^*-2}}{r^{N^*}}\right) 
+\Delta g+h(0).
\end{aligned}
\end{equation}
Since $v\in \AH_N^{m+2,p}$ implies $b_k\in H^{m+3+N^*-k,p}(\s^{d-1})$, we have 
$\Lap_\theta b_k\in H^{m+1+N^*-k,p}(\s^{d-1})$, which is more than enough for the asymptotics in 
$\Lambda v$ to be in $\AH_{2,N}^{m,p}$. Moreover, by \eqref{eq:ODEs-5},
\[
\begin{aligned}
h(0)= &\chi\left(\frac{\Lap_\theta b_{N^*-1}+(N^*-1)(N^*+1-d)b_{N^*-1}}{r^{N^*+1}}
+\frac{\Lap_\theta b_{N^*}+N^*(N^*+2-d)b_{N^*}}{r^{N^*+2}}\right) \\
&+\,2\nabla\chi\cdot\nabla \tilde{b}+(\Lap \chi)\,\tilde{b},
\end{aligned}
\]
where $\tilde{b}:=b_0+\cdots+\frac{b_{N^*}}{r^{N^*}}$. Putting this in \eqref{def:w},  we see from 
\eqref{eq:Lap(u)} at $t=0$ that $\Lambda v=\Lap v$. 
\end{proof}

Let us now complete the proof of Theorem \ref{th:SonAH}.

\begin{proof}[Proof of Theorem \ref{th:SonAH}]  
By  Proposition \ref{pr:SonAH} above, we have only to prove the last two statements of the theorem.
In this proof we follow the notation introduced in the proof of Proposition \ref{pr:SonAH}. 
In particular, for $v\in\AH^{m,p}_N$ we write $v=b+g$ as in \eqref{def:v}.
We will show that $S(t)$ as expressed in \eqref{S(t)v=} extends to an analytic semigroup of angle $\pi/2$.
To this end, denote by $S_1(t)v$, $S_2(t)v$, and $S_3(t)v$, the summands in \eqref{S(t)v=} considered in the order 
of their appearance. Then, $S(t)=S_1(t)+S_2(t)+S_3(t)$, for $t\in[0,\infty)$.

By Theorem \ref{th:SonH}, we already know that $S(t)$ is an analytic semigroup of angle $\pi/2$ on $H_{N}^{m,p}(\R^d)$, 
so the operator $S_2(t)$ on $\AH^{m,p}_N$ defined by $v\mapsto S(t)g$ extends analytically to $S_2(z)$ for 
$z\in{\bbS}_{\pi/2}\setminus\{0\}$, and for any given $\epsilon>0$ we have that $S_2(z)v\equiv S(z)g\to g$ in 
$H_{N}^{m,p}(\R^d)$ as $z\to 0$, $z\in{\bbS}_{\pi/2-\epsilon}$. 
Moreover, using \eqref{eq:S-estimate} we obtain that there exists $M\equiv M_{p,\delta,\epsilon,N}$ such that
\begin{equation}\label{eq:S2-estimate}
\|S_2(z)v\|_{\AH^{m,p}_N} = \|S(z)g\|_{H^{m,p}_{N}}\le M (1+|z|)^{N/2}\|g\|_{H^{m,p}_{N}}
\le M (1+|z|)^{N/2}\|v\|_{\AH^{m,p}_N}
\end{equation}
for all $z\in {\bbS}_{\pi/2-\epsilon}$.

To handle $S_1(t)$, let us observe from \eqref{eq:ODEs-solution-a} and \eqref{eq:ak=...} that for any 
$0\le k\le N^*$, $a_k(t)$ depends linearly on $b_0,\dots,b_k$ in such a way that 
$a_0(t)=b_0$, $a_1(t)=b_1$, and
\begin{equation}
a_k(t)= b_k+\sum_{1\le j\le k/2}t^j\,D_{jk}[b_0,\dots,b_{k-2j}]\quad\hbox{for}\ 2\le k\le N^*,
\end{equation}
where each $D_{jk}$ is a linear differential operator such that
\[
D_{jk} : H^{m+1+N^*,p}(\s^{d-1})\times\cdots\times H^{m+1+N^*-(k-2j),p}(\s^{d-1})
\to H^{m+1+N^*-(k-2j),p}(\s^{d-1})
\]
is bounded. By replacing $t$ by $z$, we then see that $z\mapsto S_1(z)$, $\C\to\mathcal{L}(\AH^{m,p}_N)$, 
is analytic for all $z\in\C$ and $S_1(z)v\to b$ as $z\to 0$ in $\AH^{m,p}_N$. Moreover, there exists a positive constant 
$M>0$ such that
\begin{equation}\label{eq:S1-estimate}
\|S_1(z)v\|_{\AH^{m,p}_N}\le M\,(1+|z|)^{N^*/2}\,\|v\|_{\AH^{m,p}_N}
\end{equation}
for all  $z\in\C$.

Finally, we consider $S_3(t)v\equiv\int_0^t S(t-s)h(s)\,ds$ where $h(t)$ is given in \eqref{eq:ODEs-5}. 
Note that $h(t)$ is a polynomial of $t$ of degree $\le N^*/2$ with coefficients in $H^{m-1,p}_N(\R^d)$. 
For any $z\in\bbS_{\pi/2}$ we can write
\begin{equation}\label{eq:S_3}
S_3(z)v=\int_0^z S(\tau)h(z-\tau)\,d\tau
\end{equation}
where the integration is taken over the straight interval in $\C$ connecting $0$ with $z$
(or any other $C^1$-curve in $\bbS_{\pi/2}$ that connects $0$ with $z$).
One then sees from \eqref{eq:DS-estimate} that the improper integral in \eqref{eq:S_3} has an integrable singularity at 
$\tau=0$ and it converges locally uniformly in $z\in S_{\pi/2}$.
This together with the analyticity statement in Corollary \ref{coro:DS-estimate} then implies that the integral in \eqref{eq:S_3} 
is a uniform limit of analytic maps with values in $H^{m,p}_N$ on bounded sets of $z\in\bbS_{\pi/2}\setminus\{0\}$.
Hence, the map $z\mapsto S_3(z)v$, $\bbS_{\pi/2}\setminus\{0\}\to H^{m,p}_N(\R^d)$, is analytic and for any given $\epsilon>0$,
\[
S_3(z)v\to 0\quad\text{in}\quad H^{m,p}_N(\R^d)\quad\text{as}\quad z\to 0,\quad\,z\in\bbS_{\pi/2-\epsilon}. 
\]
Moreover, for $z\in\bbS_{\pi/2-\epsilon}$ we obtain from \eqref{eq:DS-estimate} and \eqref{eq:h-estimate} (which also holds 
for $t$ replaced by $z$ on the left and $t$ replaced by $|z|$ on the right side of the formula) that 
\begin{equation}\label{eq:S3-estimate}
\begin{aligned}
\big\|S_3(z)v\big\|_{\AH^{m,p}_N}&\le |z|\,\int_0^1\big \|S\big((1-\tau)z\big)\,h(\tau z)\big\|_{H^{m,p}_N}\,d\tau\\
&\le M\,|z| (1+|z|)^{N/2}(1+|z|)^{N^*/2}\,\|v\|_{\AH^{m,p}_N}\le M_1\,(1+|z|)^\mu\,\|v\|_{\AH^{m,p}_N}
\end{aligned}
\end{equation}
for some positive constant $M_1>0$ depending on $1<p<0$, $\delta\in\R$, $\epsilon\in(0,\pi/2)$, and $N\ge 0$,
and for $\mu\equiv\big(N+N^*+2\big)/2$. We can now apply \cite[Theorem 3.12 in Chapter 3]{Kato} as in 
the proof of Corollary \ref{coro:DS-estimate} above to conclude that $z\mapsto S_3(z)$, 
$\bbS_{\pi/2}\setminus\{0\}\to\mathcal{L}(\AH^{m,p}_N)$, is analytic. 

Summarizing the above, we see that the map $z\mapsto S(z)v$, 
$\bbS_{\pi/2}\setminus\{0\}\to\mathcal{L}\big(\AH^{m,p}_N(\R^d)\big)$, is analytic and, 
by \eqref{eq:S2-estimate}, \eqref{eq:S1-estimate}, and \eqref{eq:S3-estimate}, it satisfies \eqref{eq:S-estimate on AH}.
We also proved that $S(z)v\to v$ in $\AH^{m,p}_N$ as $z\to 0$, $z\in\bbS_{\pi/2-\epsilon}$ for any given $\epsilon\in(0,\pi/2)$. 
The semigroup property $S(z_1+z_2)=S(z_1)S(z_2)$ on ${\bbS}_{\pi/2}$ follows from the semigroup property
on $[0,\infty)$ and the analyticity of the extension of $\{S(t)\}_{t\ge 0}$ to ${\bbS}_{\pi/2}\setminus\{0\}$. 
The last statement of Theorem \ref{th:SonAH} follows from the fact that $a_0(z)=b_0$ and $a_1(z)=b_1$.
This completes the proof of Theorem \ref{th:SonAH}.
\end{proof}

\section{Equilibrium and asymptotics in a semilinear heat flow}\label{sec:application-nonlinear}
In this section, we illustrate our previous results by applying them to the simple model equation
\begin{equation}\label{semilinear1}
\left\{
\begin{array}{l}
u_t=\Lap u+\varphi-\psi\,u^3\quad\hbox{for}\quad t>0,\ x\in\R^3,\\
u|_{t=0}=v,
\end{array}
\right.
\end{equation}
where $\varphi,\psi\in{\mathcal S}(\R^3)$ with $\psi\ge 0$. We have specified $d=3$ and $u^3$ just for convenience; 
analogous results pertain for $d\geq 3$ and $u^k$ where $k$ is any odd integer.
We will consider $v\in X$  for both $X=H^{m,p}_\delta$ and $X=\AH^{m,p}_{1,N}$, and show that the heat flow is global in time on $X$. 
We start by considering {\em mild solutions} of \eqref{semilinear1}, i.e.\ continuous solutions of the integral equation
\begin{equation}\label{semilinear1-mild}
u(t)=S(t)v+\int_0^t S(t-s)\big(\varphi-\psi u^3(s)\big)\,ds.
\end{equation}
Our first preliminary result concerns mild solutions of \eqref{semilinear1} in the weighted Sobolev space $H^{1,p}_\delta$.

\begin{Prop}\label{pr:semilinear-global-Lp} 
For $p>3$, $\delta\ge 0$, and for any $v\in H_\delta^{1,p}$ the equation \eqref{semilinear1-mild} has a unique global solution 
$u\in C([0,\infty),H_\delta^{1,p})$.
\end{Prop}

\proof
Take $v\in H^1_\delta$. Since $1>3/p$, we know by Proposition \ref{pr:H-properties}(b) that $H^{1,p}_\delta$ is 
a Banach algebra, so $F(u):=\varphi-\psi u^3$ defines an analytic map $F : H^{1,p}_\delta\to H^{1,p}_\delta$ that 
satisfies the Lipschitz condition \eqref{local-Lipschitz} in Appendix \ref{sec:mildsolutions}. 
By Theorem \ref{th:SonH}, the Laplace operator $\Delta$, considered as an unbounded operator on $H^{1,p}_\delta$
with domain $H^{3,p}_\delta$, generates an analytic semigroup $\{S(t)\}_{t\ge 0}$  on $H^{2,p}_\delta$ given by the Poisson 
integral \eqref{eq:Gaussian}. Then, by Proposition \ref{prop:analytic_semigroup_regularity}, there exists a unique solution 
\begin{equation}\label{eq:u-differentiable}
u\in C([0,T),H^{1,p}_\delta)\cap C^1((0,T),H^{1,p}_\delta)
\end{equation}
of \eqref{semilinear1} (and \eqref{semilinear1-mild}) where $T\equiv T_{\rm max}$ denotes the maximal time of existence. 
To prove that $u$ is global, we will assume that $T<\infty$ and will show that $\|u(t)\|_{H^{1,p}_\delta}$ remains bounded for $t\in[0,T)$.
Then, the global existence will follow from Proposition \ref{pr:local-existence-uniqueness}.
It follows from \eqref{eq:u-differentiable} and the Sobolev embedding $H^{1,p}_\delta\subseteq C_b(\R^d)\cap L^p_\delta(\R^d)$,
where $C_b(\R^d)$ denotes the space of bounded continuous functions, 
that we can differentiate under the integral sign and obtain that for any $t\in(0,T)$,
\begin{align}
\frac{d}{dt} \int \x^{\delta p} |u(x,t)|^{p}\,dx
&=\,p\int \x^{\delta p} |u|^{p-2}uu_t\,dx= p\int \x^{\delta p}|u|^{p-2} u\big(\Delta u+\varphi-\psi u^3\big)\,dx\nonumber\\
&\le p\int \x^{\delta p}|\varphi| u |u|^{p-2}\,dx-p\int \nabla u\cdot\nabla\big(\x^{\delta p}|u|^{p-2} u\big)\,dx\nonumber\\
&\le p\int \x^{\delta p}|\varphi| |u|^{p-1}\,dx-p\int\delta p\x^{\delta p-2}|u|^{p-2} u\, \big(x\cdot\nabla u\big)\,dx\nonumber\\
&-p\,(p-1)\int \x^{\delta p} |u|^{p-2}|\nabla u|^2\,dx,\label{eq:norm-differentiation}
\end{align}
where we used Green's identity and the fact that $\psi\ge 0$.
If $\varphi\equiv 0$ and $\delta=0$, we see that $\|u\|_{L^p}$ is non-increasing; in particular, it remains bounded.
If we no longer assume $\varphi\equiv 0$ but retain $\delta=0$, then we obtain from \eqref{eq:norm-differentiation} and H\"older's
inequality that
\begin{align*}
\frac{d}{dt} \int |u(x,t)|^{p}\,dx&\le
p\int|\varphi| |u|^{p-1}\,dx \le p\left(\int|\varphi|^{p}\,dx\right)^{\frac{1}{p}}\left(\int |u|^{p}\,dx\right)^{\frac{p-1}{p}}\\
&\le C\left(\int |u(x,t)|^{p}\,dx\right)^{\frac{p-1}{p}}.
\end{align*}
We then solve this differential inequality to show that $\|u(t)\|_{L^{p}}$ grows at most like a power of $t$.
On the other hand, for $\delta>0$ we observe that for any $\e>0$,
$\big(ux\cdot\nabla u\big)\le|ux||\nabla u|\le\frac{1}{2\e}|u|^2\x^2+\frac{\e}{2}|\nabla u|^2$, and hence
\[
\int\x^{\delta p-2}|u|^{p-2}\,\big(ux\cdot\nabla u\big)\,dx\le\frac{1}{2\e}\int\x^{\delta p}|u|^{p}\,dx+
\frac{\e}{2}\int \x^{\delta p}|\nabla u|^2|u|^{p-2}\,dx .
\]
By taking $\e:=2(p-1)/(\delta p)$ we then obtain from \eqref{eq:norm-differentiation} and H\"older's inequality
that
\[
\frac{d}{dt}\int \x^{\delta p} |u(x,t)|^{p}\,dx\le p\int \x^{\delta p}|\varphi| |u|^{p-1}\,dx 
+C\int \x^{\delta p}|u|^{p}\,dx
\]
\[
\le C_1 \left(\int \x^{\delta p}|u(x,t)|^{p}\,dx\right)^{\frac{p-1}{p}}+C\int \x^{\delta p}|u(x,t)|^{p}\,dx
\]
with constants $C, C_1>0$ depending on $1<p<\infty$ and $\delta\ge 0$.
We then use this differential inequality to show that $\|u(t)\|_{L_\delta^{p}}$ grows at most exponentially in $t$;
in particular, it remains bounded on $(0,T)$. Exactly the same arguments applied to the $t$-derivative of $\|u^3(t)\|_{L^p_\delta}^p$ 
show that there exists $C'>0$ such that
\begin{equation}\label{eq:||u^3||-bounded}
\|u^3(t)\|_{L^p_\delta}<C'\quad\text{\rm for}\quad t\in(0,T).
\end{equation}
Finally we show that $\|\nabla u(t)\|_{L^{p}_\delta}$ remains bounded for $t\in(0,T)$. From \eqref{semilinear1-mild} we can estimate 
\[
\|\nabla u(t)\|_{L^p_\delta}\le \|S(t)\nabla v\|_{L^p_\delta}+\int_0^t \| S(t-s)\nabla \varphi \|_{L^p_\delta}ds
+\int_0^t\big\|\nabla S(t-s)\big(\psi u^3(s)\big)\big\|_{L^p_\delta}\,ds.
\]
Note that by Theorem \ref{th:SonH}, $\|S(t)\nabla v\|_{L^p_\delta}\le C(1+t)^{\delta/2}\|\nabla v\|_{L^p_\delta}$ and 
\[
\int_0^t \|S(t-s)\nabla\varphi\|_{L^p_\delta}ds \le C (1+t)^{\delta/2 +1}\|\nabla \varphi\||_{L^p_\delta}.
\]
Moreover, in view of by Proposition \ref{prop:S-estimates} and \eqref{eq:||u^3||-bounded}, we also have
\begin{align*}
\int_0^t \|\nabla S(t-s)\big(\psi u^3(s)\big)\|_{L^p_\delta}\,ds\le 
C_1\int_0^t (1+t-s)^{\delta/2}(t-s)^{-1/2}\|\psi u^3(s)\|_{L^p_\delta}\,ds\\
\le C_1(1+t)^{\delta/2} \int_0^t (t-s)^{-1/2} e^{ks}\,ds\le C_2(1+t)^{\delta/2} t^{1/2} e^{kt}
\end{align*}
where $C_1,C_2>0$ are constants. In this way we see that $\|u(t)\|_{H^{1,p}_\delta}$ remains bounded for $t\in(0,T)$ 
where $T>0$ is the maximal time of existence of the solution. This implies that $u\in C([0,\infty),H_\delta^{1,p})$.
\endproof

We can improve the regularity as follows:

\begin{Prop}\label{pr:globalexistence-m+2} 
For $p>3$, $\delta\ge 0$, integer $m\ge 0$, and  any $v\in H_\delta^{m+1,p}$, the equation \eqref{semilinear1-mild} has 
a unique global solution $u\in C([0,\infty),H_\delta^{m+1,p})$.
\end{Prop}

\proof
We proceed by induction on $m$. Proposition \ref{pr:semilinear-global-Lp} provides the case $m=0$, and we will now prove 
the case $m=1$; the general induction step is analogous. 

So let us assume $v\in H_\delta^{2,p}$ and let $u\in C([0,\infty),H_\delta^{1,p})$ denote the solution obtained in 
Proposition \ref{pr:semilinear-global-Lp}. Let $q:=3\psi u^2$ and $r:=\partial\varphi-(\partial\psi)u^3$ where 
$\partial$ denotes any first order spatial (weak) derivative. 
Since $u\in C([0,\infty),H^{1,p}_\delta)$ and since $H^{1,p}_\delta$ is a Banach algebra, we have that 
$q,r\in C([0,\infty),H_\delta^{1,p})$. By applying $\partial$ to  \eqref{semilinear1-mild} and then using that
$S(t)$ commutes with $\partial$, we see that $\partial u\in C([0,\infty),L^p_\delta)$ is a mild solution of
\begin{equation}\label{linear-w}
\left\{
\begin{array}{l}
w_t=\Delta w-qw+r, \quad\hbox{for $t>0$, \ $x\in\R^3$,} \\
w|_{t=0}=\partial v.
\end{array}
\right.
\end{equation}
In particular, the non-homogeneous term $F(t,w)\equiv -q(t)w+r(t)$ in \eqref{linear-w} is globally Lipschitz in 
$H^{1,p}_\delta$ (see \eqref{global-Lipschitz}). Consequently, by Proposition \ref{pr:global-existence-uniqueness} in 
Appendix \ref{sec:mildsolutions}, \eqref{linear-w} has a unique mild solution $w\in C([0,\infty),H^{1,p}_\delta)$. 
By the uniqueness, $\partial u\equiv w\in C([0,\infty),H^{1,p}_\delta)$. Since $\partial$ is any partial derivative, we have 
$u\in C([0,\infty),H^{2,p}_\delta)$.
\endproof

By combining the results above with Proposition \ref{pr:localexistence-H} in Appendix \ref{sec:mildsolutions} we obtain
the following

\begin{Th}\label{th:globalexistence-m+2} 
For $p>3$, $\delta\ge 0$,  integer $m\ge 0$, and  any $v\in H_\delta^{m+2,p}$, the equation \eqref{semilinear1} 
has a unique global  solution $u\in C([0,\infty),H_\delta^{m+2,p})\cap C^1([0,\infty), H_\delta^{m,p})$.
\end{Th}

\proof
Let $u\in C([0,\infty),H^{m+2,p}_\delta)$ be the unique mild solution of \eqref{semilinear1} in Proposition 
\ref{pr:globalexistence-m+2}. Note that $p>3$ guarantees $m+2>3/p$, so Proposition \ref{pr:H-properties}(b) implies that 
$H^{m+2,p}_\delta$ is a Banach algebra. Consequently, $F(u)\equiv\varphi-\psi u^3$ defines an analytic map 
$F : H^{m+2,p}_\delta\to H^{m+2,p}_\delta$ that satisfies the Lipschitz condition \eqref{local-Lipschitz}. 
So we can apply Proposition \ref{pr:localexistence-H} to conclude that for any $t_0>0$ there exists a unique solution 
${\tilde u}\in C([t_0,t_0+\epsilon),H_\delta^{m+2,p})\cap C^1([t_0,t_0+\epsilon), H_\delta^{m,p})$ for some 
$\epsilon>0$ of the equation
\begin{equation}\label{eq:semilinear_shifted}
\left\{
\begin{array}{l}
u_t=\Delta u+F(u),\\
u|_{t=t_0}=u(t_0)\,.
\end{array}
\right.
\end{equation}
But this solution is also a mild solution of \eqref{eq:semilinear_shifted} on
$[t_0,t_0+\epsilon)$, so for $t\in[t_0,t_0+\epsilon)$ it coincides with the mild solution of \eqref{eq:semilinear_shifted},
\[
u(t)=S(t_0-t)u(t_0)+\int_0^{t-t_0}S(t_0-t-s) F(u(s))\,ds.
\]
In view of the arbitrariness of the choice of $t_0>0$ we then conclude that 
$u\in C([0,\infty),H_\delta^{m+2,p})\cap C^1([0,\infty), H_\delta^{m,p})$ 
and that it is a solution of \eqref{semilinear1}.
\endproof

Now we are ready to consider $v\in \AH_{1,N}^{m+2,p}$ in \eqref{semilinear1}, in other words to prove Theorem \ref{th:semilinear-global}.

\begin{proof}[Proof of Theorem \ref{th:semilinear-global}]
The hypotheses imply $\AH^{m+2,p}_{1,N}\subseteq H^{m+2,p}$, so we can apply Theorem \ref{th:globalexistence-m+2} to 
find a (unique) solution of \eqref{semilinear1} satisfying $u\in C([0,\infty),H^{m+2,p})\cap C^1([0,\infty),H^{m,p})$. 
We want to show that $u(t)\in \AH^{m+2,p}_{1,N}$ for all $t>0$ and 
$u\in C([0,\infty),\AH^{m+2,p}_{1,N})\cap C^1([0,\infty),\AH^{m,p}_{1,N})$.
Since $p>3$ we have $N^*\equiv N$ and we can write 
\[
v(x)=b(x)+g(x)=\chi(r)\left(\frac{b_1(\theta)}{r}+\cdots+\frac{b_{N}(\theta)}{r^{N}}\right)+g(x),
\]
where $b_k\in H^{m+3+N-k,p}(S^{2})$ and $g\in H^{m+2,p}_{N}$.
We seek $u$ of the form
\begin{equation}\label{eq:u=a+f}
u(x,t):=a(x,t)+f(x,t)=\chi(r)\left(\frac{a_1(\theta,t)}{r}+\cdots+\frac{a_{N}(\theta,t)}{r^{N}}\right)+f(x,t)
\end{equation}
with initial conditions $a_k(\theta,0)=b_k(\theta)$, $0\le k\le N$, and $f(x,0)=g(x)$, and such that
\begin{eqnarray*}
&f\in C([0,\infty),H^{m+2,p})\cap C^1((0,\infty),H^{m,p}),\\
&a_k\in C([0,\infty),H^{m+2,p}(S^2))\cap C^1((0,\infty),H^{m,p}(S^2)),\,\,1\le k\le N.
\end{eqnarray*}
Plugging $u$ into $u_t=\Lap u+\varphi-\psi u^3$ we obtain 
\begin{equation}\label{semilinear1'}
a_t+f_t=\Lap a+\Lap f+\varphi-\psi u^3\,.
\end{equation}
Since $u\in C([0,\infty),H^{m+2,p})\cap C^1([0,\infty),H^{m,p})$ satisfies \eqref{semilinear1}, 
we see from \eqref{eq:Lap(u)}, \eqref{eq:u_dot}, and \eqref{semilinear1'} that if $a$ satisfies the system of equations:
\begin{equation}\label{eq:a}
\begin{aligned}
&\dt{a}_1(t)=\dt{a}_2(t)=0 \Rightarrow a_1(t)=b_1,\,a_2(t)=b_2,\\
&\dt{a}_3(t)=\Delta_\theta a_1(t)\\
&\dt{a}_4(t)=\Delta_\theta a_2(t)+2a_2(t),\\
&\cdots\\
&\dt{a}_{N}(t)=\Delta_\theta a_{N-2}(t)+(N-2)(N-3) a_{N-2}(t),
\end{aligned}
\end{equation}
with initial conditions $a_k(0)=b_k$, $1\le k\le N$, then $f(t):=u(t)-a(t)$ satisfies
\begin{equation}\label{eq:f}
\left\{
\begin{array}{l}
\dt{f}(t)=\Delta f+h(t)+\varphi-\psi\,u(t)^3,\\
f|_{t=0}=g,
\end{array}
\right.
\end{equation}
where $h(t)$ is given by \eqref{eq:ODEs-5} and where we have again used a dot to denote $t$-differentiation. 
We can successively integrate \eqref{eq:a} and obtain that for any $1\le k\le N$, $a_k(t)$ ia a polynomial
of $t$ of degree $\le (k-1)/2$ with coefficients in $H^{m+3+N-k}(S^2)$.
This together with \eqref{eq:u=a+f} then implies that
\begin{equation}\label{eq:a-regularity}
a\in C([0,\infty), \AH^{m+2,p}_{1,N})\cap C^1([0,\infty), \AH^{m,p}_{1,N}).
\end{equation}
Let us now turn our attention to equation \eqref{eq:f}, that can be written as
\begin{equation}\label{eq:f'}
\left\{
\begin{array}{l}
\dt{f}(t)=\Delta f+A(t),\quad A(t):=h(t)+\varphi-\psi u(t)^3\\
f|_{t=0}=g,
\end{array}
\right.
\end{equation}
where $u\in C([0,\infty),H^{m+2,p})\cap C^1([0,\infty),H^{m,p})$
and
\[
h\in C^1([0,\infty),H^{m+1,p}_N)\subseteq C^1([0,\infty),H^{m,p}_N)
\] 
by \eqref{eq:ODEs-5}. Since $\varphi,\psi\in\Sz$ we then obtain that
\begin{equation}\label{eq:A}
A\in C^1([0,\infty),H^{m,p}_N).
\end{equation}
It then follows from Theorem \ref{th:SonH} and the last statement of 
Proposition \ref{pr:local-existence-uniqueness} in Appendix \ref{sec:mildsolutions} that
\[
f\in C([0,\infty),H^{m+2,p}_N)\cap C^1([0,\infty),H^{m,p}_N).
\]
This together with \eqref{eq:a-regularity} shows that
$u\in C([0,\infty),\AH_{1,N}^{m+2,p})\cap C^1([0,\infty),\AH_{1,N}^{m,p})$.
This completes the proof of the theorem.
\end{proof}


Let us now prove Theorem \ref{th:semilinear-equilibrium} concerning solutions of  the semilinear elliptic equation
\begin{equation}\label{eq:semilinear-equilibrium1}
\Delta u=\psi\, u^3-\varphi\quad\hbox{for}\quad x\in\R^3.
\end{equation}
More specifically, we will prove that for any $p>3$, integer numbers $m,N\ge 1$, and for any $\psi,\varphi\in\mathcal{S}(\R^3)$ with 
$\psi\ge 0$ the semilinear elliptic equation \eqref{eq:semilinear-equilibrium1} has a unique solution $u_*\in\AH^{m+2,p}_{1,N}$.
We will also prove that the solution $u_*$ has non-trivial asymptotic coefficients generically.

\begin{proof}[Proof of Theorem \ref{th:semilinear-equilibrium}]
In this proof we use the weighted Sobolev spaces $W^{m,p}_\delta$ and the asymptotic spaces $\A^{m,p}_{n,N}$ that 
are defined in Appendix \ref{sec:appendix_properties}. In particular, for integer $m\ge 1$, 
let ${\mathcal I}^{m,p}_{1,N}$ denote the subspace of $\A^{m,p}_{1,N}$ consisting of functions $u$ of the form
\begin{equation*}
u(x)=\chi(r)\left(\frac{a_1(\theta)}{r}+\cdots+\frac{a_{N}(\theta)}{r^{N}}\right)+f(x), \quad f\in W^{m,p}_{\gamma_N},
\end{equation*}
where $a_k$ is an eigenfunction for the Laplace-Beltrami operator $-\Delta_S$ on the unit sphere $S^2$ with
eigenvalue $\lambda_{k-1}=k(k-1)$ and $\gamma_N\equiv\gamma_0+N$ with 
\begin{equation}\label{eq:gamma_range}
0<\gamma_0+3/p<1. 
\end{equation}
Since the eigenspace ${\mathcal H}_{k-1}$ for $\lambda_{k-1}$ is finite-dimensional and 
consists of smooth functions, ${\mathcal I}^{m,p}_{1,N}$ is a closed subspace of $\A^{m,p}_{1,N}$. 
Since $m\ge 1>3/p$ for $p>3$ we conclude that $\A^{m,p}_{1,N}$ and $W^{m,p}_{\gamma_N}$ are 
Banach algebras (cf.\ \cite{McOwenTopalov2}). Moreover, a calculation in spherical coordinates shows 
$\Delta\big(a_k(\theta)/r^k\big)=0$ for $r>0$, so it is clear that 
\begin{equation} \label{Laplacian-I}
\Delta: {\mathcal I}^{\ell+2,p}_{1,N} \to W^{\ell,p}_{\gamma_N+2}
\end{equation}
for any integer $\ell\ge 0$. In fact, by Proposition B.1 in \cite{McOwenTopalov4}, \eqref{Laplacian-I} is an isomorphism.
We will also need the fact that (see \cite[Theorem 5.1]{Lockhart},\cite[Lemma 2.1]{Choquet-Bruhat-C})
\begin{equation}\label{W-compactness}
W^{m,p}_\delta\subseteq W^{m_1,p}_{\delta_1} \quad\hbox{is a compact embedding for}\ m>m_1\ \hbox{and}\ \delta>\delta_1.
\end{equation}

\medskip

\noindent{\em The map $\F$:} Take $\psi\in\Sz$ such that $\psi\ge 0$ and consider the map
\begin{equation}\label{eq:FF}
\F : {\mathcal I}^{m+2,p}_{1,N}\to W^{m,p}_{\gamma_N+2},\quad u\mapsto\Delta u-\psi\,u^3.
\end{equation}
Since $\A^{m+2,p}_{1,N}$ is a Banach algebra and $\psi\in\Sz$ the map \eqref{eq:FF} is continuously differentiable 
and for any given $u\in {\mathcal I}^{m+2,p}_{1,N}$ we have that
\begin{equation}\label{eq:FF-differentiable}
d_u\F(w)=\Delta w-3\psi\,u^2 w
\end{equation}
where $d_u\F : {\mathcal I}^{m+2,p}_{1,N}\to W^{m,p}_{\gamma_N+2}$ is the differential of \eqref{eq:FF} at $u$.

\medskip

\noindent{\em The local diffeomorphism property of $\F$}:
Note that the map
\[
K_u : {\mathcal I}^{m+2,p}_{1,N}\to W^{m,p}_{\gamma_N+2},\quad  w\mapsto-3\psi\,u^2 w,
\]
can be decomposed as ${\mathcal I}^{m+2,p}_{1,N}\to W^{m+2,p}_{\gamma_N+3}\subseteq W^{m,p}_{\gamma_N+2}$
which is compact in view of \eqref{W-compactness}. Since $\Delta : {\mathcal I}^{m+2,p}_{1,N}\to W^{m,p}_{\gamma_N+2}$
is an isomorphism, we then conclude that the differential \eqref{eq:FF-differentiable} is a Fredholm operator of index zero.
We will now prove that the kernel of \eqref{eq:FF-differentiable} is trivial, and hence,  \eqref{eq:FF-differentiable} is an isomorphism.
To this end, assume that $w\in{\mathcal I}^{m+2,p}_{1,N}$ satisfies the elliptic equation
\[
\Delta w=3\psi\,u^2 w.
\]
for some $u\in{\mathcal I}^{m+2,p}_{1,N}$. 
Since $u,w\in{\mathcal I}^{m+2,p}_{1,N}$ for $m\geq 1$, we conclude that $u,w\in C^2(\R^3)$ and that
\begin{equation}\label{eq:w_decay}
w(x)=O(1/r)\quad\text{\rm as}\quad|x|\to\infty.
\end{equation}
For a given $R>0$ consider the open ball $\Omega_R:=\big\{x\in\R^3\,:\,|x|<R\big\}$ in $\R^3$
as well as the open sets
\[
\Omega_R^+:=\big\{x\in\R^3\,:\,w(x)>0\big\}\quad\text{\rm and}\quad\Omega_R^-:=\big\{x\in\R^3\,:\,w(x)<0\big\}.
\]
If $\Omega_R^+$ is not empty we have that 
\[
\Delta w=3\psi\,u^2 w\ge 0\quad \hbox{in}\quad{\Omega_R^+}.
\]
Then, $\max_{\overline{\Omega}_R} w=\max_{\overline{\Omega}_R^+} w\le\max_{|x|=R}|w|$, where we have used 
the maximum principle on $\Omega_R^+$ and the fact that $w\equiv 0$ on the boundary of $\Omega_R^+$ in $\Omega_R$.
Hence $\max_{\overline{\Omega}_R} w\le\max_{|x|=R}|w|$, which also holds trivially in the case when $\Omega_R^+$ is empty.
The same argument applied to $(-w)$ on $\Omega_R^-$ then implies that
\begin{equation}\label{eq:w-estimate}
\max_{\overline{\Omega}_R}|w|\le\max_{|x|=R}|w|\,.
\end{equation}
By taking $R\to\infty$ in \eqref{eq:w-estimate} we then conclude from \eqref{eq:w_decay} that $w\equiv 0$. This proves that 
the differential $d_u\F : {\mathcal I}^{m+2,p}_{1,N}\to W^{m,p}_{\gamma_N+2}$ is an isomorphism of Banach spaces. 
Hence, by the inverse function theorem, the map \eqref{eq:FF} is a {\em local diffeomorphism}.

\medskip

\noindent{\em The properness of $\F$:}
Our next task is to prove that the map \eqref{eq:FF} is proper, which will imply that \eqref{eq:FF} is a diffeomorphism.
Recall that a map $F : X\to Y$ between two metric spaces is called {\em proper} if the preimage of any compact set in $Y$ is compact
in $X$. In order to show that $\F$ is proper, take a sequence $(u_j)_{j\ge 1}$ in ${\mathcal I}^{m+2,p}_{1,N}$ such 
that $\F(u_j)$ converges in $W^{m,p}_{\gamma_N+2}$ to some $f\in W^{m,p}_{\gamma_N+2}$:
\begin{equation}\label{eq:properness_assumption}
f_j:=\F(u_j)\to f \quad \hbox{in}\quad W^{m,p}_{\gamma_N+2}\ \text{\rm as}\  j\to\infty .
\end{equation}
We want to show that $(u_j)_{j\ge 1}$ contains a subsequence that converges in ${\mathcal I}^{m+2,p}_{1,N}$.
For any $j\geq 1$, we use $\Delta^{-1} : W^{m,p}_{\gamma_N+2}\to{\mathcal I}^{m+2,p}_{1,N}$, i.e.\ 
the inverse of the isomorphism  \eqref{Laplacian-I}, to define  $\tilde{f}_j:=\Delta^{-1} f_j$ and $\tilde{f}:=\Delta^{-1} f$, 
and hence obtain
\begin{equation}\label{eq:properness_relation}
\left\{
\begin{array}{l}
\Delta\big(u_j+\tilde{f}_j\big)=\psi u_j^3,\\
\tilde{f}_j\to \tilde{f}\quad \text{in}\quad{\mathcal I}^{m+2,p}_{1,N} \ 
\text{\rm as}\ j\to\infty.
\end{array}
\right.
\end{equation}
In what follows we will again use the maximum principle. Fix $j\geq 1$, take $R>0$, and consider the open sets $\Omega_R$ and 
$\Omega_R^\pm$ in $\R^3$ defined as above but with $w$ replaced by $u_j$. Assume that $\Omega_R^+$ is not empty. 
It then follows from \eqref{eq:properness_relation} that
\[
\Delta\big(u_j+\tilde{f}_j\big)=\psi u_j^3\ge 0\quad \text{in}\quad{\Omega_R^+}.
\]
By the maximum principle, 
\begin{equation}\label{eq:maximum_principle2}
\max_{\overline{\Omega}_R^+}\big(u_j+\tilde{f}_j\big)
\le\max_{|x|=R}|u_j|+\max_{\overline{\Omega}_R}|\tilde{f}_j|
\end{equation}
where we used that $u_j\equiv 0$ on the boundary of $\Omega_R^+$ in $\Omega_R$ and the inequality
$\max_{\partial\big(\overline{\Omega}_R^+\big)}|\tilde{f}_j|\le\max_{\overline{\Omega}_R}|\tilde{f}_j|$.
Since
$\max_{\overline{\Omega}_R^{-}}\big(u_j+\tilde{f}_j\big)\le\max_{|x|=R}|u_j|+\max_{\overline{\Omega}_R}|\tilde{f}_j|$,
we obtain that
\[
\max_{\overline{\Omega}_R}\big(u_j+\tilde{f}_j\big)\le\max_{|x|=R}|u_j|+\max_{\overline{\Omega}_R}|\tilde{f}_j|.
\]
This estimate also trivially holds in the case when $\Omega_R^+$ is empty.
The same argument applied to $(-u_j)$ and $(-\tilde{f}_j)$ on $\Omega_R^-$ then implies that
\begin{equation}\label{eq:u-estimate}
\max_{\overline{\Omega}_R}\big|u_j+\tilde{f}_j\big|\le\max_{|x|=R}|u_j|+\max_{\overline{\Omega}_R}|\tilde{f}_j|\,.
\end{equation}
By taking $R\to\infty$ we then see that
$\|u_j+\tilde{f}_j\|_{\infty}\le\|\tilde{f}_j\|_{\infty}$ 
where $\|\cdot\|_{\infty}$ denotes the sup-norm in $L^\infty(\R^3)$.
Hence
\[
\|u_j\|_{\infty}\le 2\|\tilde{f}_j\|_{\infty}\le C
\]
uniformly in $j\ge 1$, where we used the second relation in \eqref{eq:properness_relation}
to conclude that $(\tilde f_j)_{j\ge 1}$ is bounded in ${\mathcal I}^{\ell+2,p}_{1,N}\subseteq L^\infty$.
This implies that
\[
\|\psi u_j^3\|_{L^p_{\gamma_N +2}}\le C^3\|\psi\|_{L^p_{\gamma_N +2}}
\]
uniformly in $j\ge 1$. Then, it follows from \eqref{eq:properness_relation} and the fact that \eqref{Laplacian-I} is 
an isomorphism for $\ell=0$ that $u_j$ is bounded in ${\mathcal I}^{2,p}_{1,N}$ uniformly in $j\ge 1$.
Since ${\mathcal I}^{\ell+2,p}_{1,N}\subseteq\A^{\ell+2,p}_{1,N}$ and $\A^{\ell+2,p}_{1,N}$ is a Banach algebra for $\ell\ge 0$, 
the map ${\mathcal I}^{\ell+2,p}_{1,N}\to W^{\ell+2,p}_{\gamma_N+2}$, $u\mapsto\psi u^3$, $\psi\in\Sz$, is bounded.
This shows that we can iterate this process and conclude from \eqref{eq:properness_relation} that $u_j$ is bounded in 
${\mathcal I}^{m+2,p}_{1,N}$ uniformly in $j\ge 1$. Note that the map
\[
{\mathcal I}^{m+2,p}_{1,N}\to W^{m,p}_{\gamma_N+2},\quad u\mapsto\psi u^3,
\]
is a compact nonlinear map since it can be decomposed as 
${\mathcal I}^{m+2,p}_{1,N}\to W^{m+2,p}_{\gamma_N+3}\subseteq W^{m,p}_{\gamma_N+2}$
that is a composition of a bounded non-linear map and a compact embedding (see \eqref{W-compactness}).
This together with the uniform boundedness of $u_j$ in ${\mathcal I}^{m+2,p}_{1,N}$ implies that there exist a subsequence 
$\big(u_{j_k}\big)_{k\ge 1}$ of $(u_j)_{j\ge 1}$ and $v\in W^{m,p}_{\gamma_N+2}$ such that $\psi u_{j_k}^3\to v$ in 
$W^{m,p}_{\gamma_N+2}$ as $k\to\infty$. 
We then conclude from $\Delta u_{j_k}=\psi u_{j_k}^3-f_{j_k}$ and the fact that \eqref{Laplacian-I} is an isomorphism that there exists 
$u\in{\mathcal I}^{m+2,p}_{1,N}$ such that
\[
u_{j_k}\to u\quad \text{in}\quad{\mathcal I}^{m+2,p}_{1,N} \ 
\text{\rm as}\  k\to\infty.
\]
Hence, the map \eqref{eq:FF} is proper, and therefore, it is onto. Moreover, by the Hadamard-Levi theorem 
(see e.g.\ \cite[Theorem 2.5.17, Remark (i)]{AMR}), \eqref{eq:FF} is a diffeomorphism, so for every $\varphi\in\Sz$, there is a unique 
$u_*\in {\mathcal I}^{m+2,p}_{1,N}$ satisfying $\F(u_*)=\varphi$.
Finally, we need to confirm that $u_*\in\AH^{m,p}_{1,N}$. Notice that $p>3$ implies that we can take $N^*=N$ in 
\eqref{def:N*}. In fact, $p>3$ also implies that we can choose $\gamma_0\geq 0$ in our definition of $\A^{m,p}_{1,N}$  
(see \eqref{eq:gamma_range} above and Definition \ref{def:W,A}(b) in Appendix \ref{sec:appendix_properties}), 
so $\gamma_N=\gamma_0+N\ge N$, which implies $W^{m,p}_{\gamma_N}\subseteq H^{m,p}_N$. 
Putting these together, we see that $\A^{m,p}_{1,N}\subseteq \AH^{m,p}_{1,N}$, and hence  $u_*\in\AH^{m,p}_{1,N}$. 
This proves the first statement of the theorem. 

\medskip

\noindent{\em Non-vanishing of the asymptotic coefficients:}
Let us now prove the second statement of Theorem \ref{th:semilinear-equilibrium}. 
It follows from the first statement that for any given $m,N\ge 1$ and $\ell\ge 0$ the map
\begin{equation}\label{eq:FFell}
\F : {\mathcal I}^{m+\ell+2,p}_{1,N+\ell}\to W^{m+\ell,p}_{\gamma_N+\ell+2}
\end{equation} 
is a diffeomorphism. Consider the sets
\[
\mathcal{Z}^\ell_k:=\big\{u\in{\mathcal I}^{m+\ell+2,p}_{1,N+\ell}\,:\,a_k=0\big\},\quad  1\le k\le N.
\]
It follows from the definition of the space ${\mathcal I}^{m+\ell+2,p}_{1,N+\ell}$ that for any given $1\le k\le N$ the set 
$\mathcal{Z}^\ell_k$ is a closed subspace of co-dimension equal to $\dim\mathcal{H}_{k-1}\ge 1$. 
Since \eqref{eq:FFell} is a diffeomorphism, we then conclude that for any $1\le k\le N$ the set $\F\big(\mathcal{Z}^\ell_k\big)$ is 
a submanifold of finite co-dimension in $W^{m+\ell+2,p}_{\gamma_N+\ell+2}$. This shows that for any $\ell\ge 0$ the set
\[
W^{m+\ell,p}_{\gamma_N+\ell+2}\setminus\Big(\bigcup_{k=1}^N\F\big(\mathcal{Z}^\ell_k\big)\Big)
\]
is open and dense in $W^{m+\ell,p}_{\gamma_N+\ell+2}$. The second statement of the theorem now follows from Baire's theorem,
the completeness of the Schwartz space $\Sz$, as well as the fact that $\Sz$ is dense in $W^{m+\ell,p}_{\gamma_N+2+\ell}$, $\ell\ge 0$, 
and that the Fr\'echet topology on $\Sz$ is induced by the norms in $W^{m+\ell,p}_{\gamma_N+2+\ell}$, $\ell\ge 0$.
\end{proof}

Note that asymptotic spaces with log terms (cf.\ \cite{McOwenTopalov2}, \cite{McOwenTopalov3}) can be used to prove a {\em local} version of Theorem \ref{th:semilinear-global} 
and Theorem \ref{th:semilinear-equilibrium} with a nonlinear term of the form $F(u)+\varphi$ where $F(u)$ is an arbitrary 
polynomial of $u$ vanishing at $u=0$ and with $\varphi$ in an open neighborhood of zero in $\Sz(\R^3)$.

\appendix

\section{Background on weighted and asymptotic spaces}\label{sec:appendix_properties}

Let us summarize some of the essential properties of the weighted Sobolev and asymptotic spaces defined in 
Section \ref{sec:introduction}. For proofs of Propositions \ref{pr:H-properties}-\ref{pr:A-products}, see \cite{McOwenTopalov2}.

\begin{Prop}\label{pr:H-properties}  
\begin{enumerate}
\item[(a)] For $m>d/p$ and $\delta\in\R$, if $f\in H^{m,p}_\delta(\R^d)$, then $f\in C^k(\R^d)$ for all $k<m-(d/p)$ and
\[
\sup_{x\in\R^d} \x^\delta|\partial^\alpha f(x)|\le C_{d,m,p,k,\delta}\,\|f\|_{H^{m,p}_\delta} \quad\hbox{for $|\alpha|\le k$}.
\]
In fact, for all $|\alpha|\le k$ we have
$|x|^\delta\,|\partial^\alpha f(x)|\to 0 $  as $|x|\to\infty.$
\item[(b)] If $m>d/p$ and $\delta\ge 0$, then $H^{m,p}_\delta(\R^d)$ is a Banach algebra:
\[
\|fg\|_{H^{m,p}_\delta}\le C\,\|f\|_{H^{m,p}_\delta}\,\|g\|_{H^{m,p}_\delta} \quad\hbox{for} \ f,g\in H^{m,p}_\delta(\R^d).
\]
\end{enumerate}
\end{Prop}

\begin{Rem}
If $X$ is a Banach algebra and $\gamma$ is a positive integer, then it is easy to see that $F(u):=u^\gamma$ defines an analytic 
map $X\to X$. This fact is used with $X=H^{m,p}_\delta(\R^d)$ in the proofs of Proposition \ref{pr:semilinear-global-Lp} and 
Theorem \ref{th:globalexistence-m+2}.
\end{Rem}

\begin{Prop}\label{pr:A-properties}  
\begin{enumerate}
\item[(a)] If $n_1\ge n$ and $N_1\ge N$, then  ${\AH}_{n_1,N_1}^{m,p}\subseteq {\AH}_{n,N}^{m,p}$ is 
a continuous inclusion.
\item[(b)]  If $m\ge 1$,  then $u\mapsto \partial u/\partial x_j$ is a bounded linear map 
${\AH}_{n,N}^{m,p}\to {\AH}_{n+1,N}^{m-1,p}$.
\item[(c)] Multiplication by $\chi(r)/r^k$ is bounded 
${\AH}_{n,N}^{m,p}\!\to\! {\AH}_{n+k,N+k}^{m,p}$.
\item[(d)] Assume $m>d/p$. If  $u\in {\AH}_{n,N}^{m,p}$, then
\begin{equation}\label{AH-derivatives}
\sup_{x\in\R^d} \x^n \, |\partial^\alpha u(x)|\le C\,\|u\|_{{\AH}_{n,N}^{m,p}} \quad\hbox{for}\ |\alpha|<m-d/p.
\end{equation}
\item[(e)] If $m>d/p$, then ${\AH}_{n,N}^{m,p}$ is a Banach algebra:
\begin{equation}\label{est:BanachAlgebra}
\|u\,v\|_{{\AH}_{n, N}^{m,p} }\le C\,\|u\|_{{\AH}_{n,N}^{m,p} }\,\|v\|_{{\AH}_{n,N}^{m,p} }
\quad\hbox{for}\ u,v\in {\AH}_{n,N}^{m,p}.
\end{equation}
\end{enumerate}
\end{Prop}

\noindent The property (e) in Proposition \ref{pr:A-properties} is actually a special case of a more general result:

\begin{Prop}\label{pr:A-products} 
For $m>d/p$ and $0\le n_i\le N_i$ for $i=1,2$,  let $\widetilde n=n_1+n_2$ and $\widetilde N=\min(N_1+n_2,N_2+n_1)$. Then
\begin{equation}\label{est:AH-multiplication}
\|u\,v\|_{{\AH}_{\widetilde n,\widetilde N}^{m,p} }\le C\,\|u\|_{{\AH}_{n_1,N_1}^{m,p} }\|v\|_{{\AH}_{n_2,N_2}^{m,p}}
\quad\hbox{for}\ u\in {\AH}_{n_1,N_1}^{m,p},\ v\in {\AH}_{n_2,N_2}^{m,p}\,.
\end{equation}
\end{Prop}

We use in Sections \ref{sec:application-nonlinear} a different family of weighted Sobolev and asymptotic 
spaces denoted $W^{m,p}_\delta$ and $\A^{m,p}_{N}$.  These spaces were introduced in \cite{McOwenTopalov2} and used in 
\cite{McOwenTopalov3,McOwenTopalov4} to study the Euler equations because the spaces $W^{m,p}_\delta$ have nice 
mapping properties for the Laplacian. On the other hand, the space $H^{m,p}_\delta$ has better mapping properties for 
the resolvent of the Laplacian, which is why we have used $H^{m,p}_\delta$ and $\AH^{m,p}_{N}$ predominantly in this paper.

\begin{Def} \label{def:W,A}
Let $m\in\Z_{\ge 0}$ and $1\le p<\infty$.
\begin{enumerate}
\item[\rm(a)] For $\delta\in \R$, define the weighted Sobolev space $W_\delta^{m,p}(\R^d)$ to be 
the closure of $C^\infty_c(\R^d)$ in the norm
\begin{equation}\label{eq:W-norm}
\|f\|_{W^{m,p}_\delta}=\sum_{|\alpha|\le m} \|\x^{\delta+|\alpha|}\partial^\alpha f\|_{L^p}.
\end{equation}
\item[\rm(b)] Fix $\gamma_0$ satisfying $0<\gamma_0+d/p<1$ and for $N\in \Z_{\ge 0}$ let $\gamma_N=\gamma_0+N$. 
Define $\A^{m,p}_N$ to be the space of functions of the form 
\begin{equation} \label{A-asymptotics}
v(x)=\chi(r)\left(a_0(\theta)+\cdots + \frac{a_{N}(\theta)}{r^{N}}\right)+f(x),
\end{equation}
where $a_k\in H^{m+1+N-k,p}(\s^{d-1})$ and $f\in W^{m,p}_{\gamma_N}(\R^d)$.
The function space $\A^{m,p}_{N}$ becomes a Banach space under the norm
\begin{equation}\label{def:A-norm}
\|v\|_{{\A}_{N}^{m,p}}=
\sum_{k=0}^{ N} \|a_k\|_{H^{m+1+N-k,p}} + \| f\|_{W_{\gamma_N}^{m,p}}.
\end{equation}
We denote by $\A^{m,p}_{n,N}$ the closed subspace for which $a_0=\cdots a_{n-1}=0$.
\end{enumerate}
\end{Def}
\noindent
These function spaces enjoy properties analogous to those in Propositions \ref{pr:H-properties}, \ref{pr:A-properties}, 
and \ref{pr:A-products}; for details, see \cite{McOwenTopalov2}. In particular, if $m>d/p$, then $f\in W^{m,p}_{\gamma_N}$ 
implies that $f(x)=o(r^{-N})$ as $r\to\infty$, so we generally make this assumption when dealing with 
$\A^{m,p}_{N}$ (as we did for $\AH^{m,p}_{N}$).

\section{Auxiliary lemmas and discussion}\label{sec:aux-lemmas}
The following lemma is well known. We provide the proof for convenience.

\begin{Lem}\label{Le:elementary}
For any $\delta\in\R$ there is a constant $C=C(\delta)=2^{|\delta|/2}$ such that
\begin{equation}\label{eq:elementary}
\frac{\x^\delta}{\y^\delta}\le C\langle x-y\rangle^{|\delta|}\quad \hbox{for all}\ x,y\in\R^d.
\end{equation}
\end{Lem}

\begin{proof}
By the triangle inequality, we have $1+|y+z|^2\le 2(1+|y|^2)(1+|z|^2)$ for all $y,z\in\R^d$. If $\delta>0$, this implies
$\langle y+z\rangle^\delta\le 2^{\delta/2}\y^\delta \z^\delta$, and we let $z=x-y$ to obtain \eqref{eq:elementary}.
If $\delta<0$, the triangle inequality with $x$ in place of $y$ implies 
$\langle x+z\rangle^{-\delta}\le 2^{-\delta/2}\x^{-\delta}\z^{-\delta}$, 
and we let $z=y-x$ to obtain \eqref{eq:elementary}.
\end{proof}

\begin{Lem}\label{Le:apriori-est-1}
For any $\delta\in\R$ and $1<p<\infty$, there is a constant $C>0$ so that
\begin{equation}\label{est:apriori1}
\|g\|_{H_\delta^{2,p}}\le C\left(\|\Lap g\|_{L_\delta^{p}}+\|g\|_{L_\delta^{p}}\right)
\quad\hbox{for $g\in H^{2,p}_\delta(\R^d)$.}
\end{equation}
Moreover, if $g\in L^p_\delta(\R^d)$ has distributional derivaties satisfying $\Lap g\in L^p_\delta(\R^d)$, then 
$g\in H^{2,p}_\delta(\R^d)$.
\end{Lem}

\begin{proof} 
Using the isomorphism $J_\delta:H^{m,p}_\delta(\R^d)\to H^{m,p}$ as in Lemma \ref{le:J_delta-H},  
we see that to prove \eqref{est:apriori1} it suffices to show
\begin{equation}\label{est:apriori2}
\|f\|_{H^{2,p}}\le C\left(\|{\cal L}_\delta f\|_{L^p}+\|f\|_{L^p}\right) \quad\hbox{for all $f\in H^{2,p}(\R^d)$},
\end{equation}
where ${\cal L}_\delta:H^{2,p}(\R^d)\to L^p(\R^d)$ is the elliptic operator defined by
\begin{equation}\label{eq:L_delta}
{\cal L}_\delta \, f:=J_\delta\Lap J_{-\delta}\,f
=\Lap f-2\delta\x^{-2}\sum_{j=1}^d x_j\partial_j f+\left(\delta(\delta+2)|x|^2\x^{-4}-\delta d\x^{-2}\right)f.
\end{equation}
Since ${\cal L}_\delta$ is a lower-order perturbation of $\Lap$ (with coefficients vanishing sufficiently fast as $|x|\to\infty$), 
we know by the $L^p$-boundedness of pseudo-differential operators (see e.g.\ Theorem 9.4 in \cite{AmannHieberSimonett}) that 
for sufficiently large $\mu>0$,  $({\cal L}_\delta +\mu)^{-1}:L^p\to H^{2,p}$ is well-defined and bounded. But this means
\begin{equation}\label{est:apriori3}
\| f\|_{H^{2,p}}\le C\,\|({\cal L}_\delta +\mu)f\|_{L^p}
\quad\hbox{for all $f\in H^{2,p}(\R^d)$,}
\end{equation}
from which \eqref{est:apriori2} easily follows. Finally, if $g\in L_\delta^p(\R^d)$ satisfies $\Lap g\in L_\delta^p(\R^d)$,
then $f=J_\delta g\in L^p(\R^d)$ satisfies ${\cal L}_\delta f\in L^p(\R^d)$. Hence 
$h:=({\cal L}_\delta +\mu)f\in L^p(\R^d)$, which implies $f=({\cal L}_\delta+\mu)^{-1}h\in H^{2,p}(\R^d)$, and hence
$g\in H^{2,p}_\delta(\R^d)$, as desired.
\end{proof}

\medskip

By applying \eqref{est:apriori1} to $\partial^\beta g$ for $|\beta|\le m$ with $g\in C^\infty_c(\R^d)$ and then taking the closure 
in the weighted Sobolev norm, we obtain the following

\begin{Coro}\label{Co:apriori-estimates}
For any $m\in\Z_{\ge 0}$, $\delta\in\R$, and $1<p<\infty$, there is a constant $C>0$ so that
\begin{subequations}\label{est:aprior}
\begin{equation}\label{est:aprior2}
\|g\|_{H_\delta^{m+2,p}}\le C\left(\|\Lap g\|_{H_\delta^{m,p}}+\|g\|_{H_\delta^{m,p}}\right)
\quad\hbox{for  $g\in H_\delta^{m+2,p}(\R^d)$.}
\end{equation}
Moreover, if $g\in H^{m,p}_\delta(\R^d)$ and $\Delta g\in H^{m,p}_\delta(\R^d)$ then
$g\in H^{m+2,p}_\delta(\R^d)$.
\end{subequations}
\end{Coro}

\noindent This estimate implies that $\Lap$ considered as an unbounded operator on $H^{m,p}_\delta(\R^d)$ with domain 
$D=H^{m+2,p}_\delta(\R^d)$ is a closed operator.

\medskip\medskip

Let us now discuss Remark \ref{Rem:NotClosed} in Section \ref{sec:asymptotic_spaces}. 
To show that $\Lap$ as an unbounded operator on $\AH_N^{m,p}$ with domain $D=\AH_N^{m+2,p}$ is not closed, 
we will construct $u\in \AH_N^{m,p}\backslash \AH_N^{m+2,p}$ such that $\Lap u\in \AH_N^{m,p}$; 
since $\AH_N^{m+2,p}$ is dense in $\AH_N^{m,p}$, this would enable us to find $u_j\in\AH_N^{m+2,p}$ such that 
$u_j\to u$  in $\AH_N^{m,p}$ and $\Lap u_j\to \Lap u$ in $\AH_N^{m,p}$ but $u\not\in\AH_N^{m+2,p}$. 
In fact, take 
\[
u(x)=\chi\Big(a_0(\theta)+\cdots\frac{a_{N^*}(\theta)}{r^{N^*}}\Big)+f(x)
\]
such that $f\in H^{m+2,p}_N$ and 
\begin{align}
&a_k\in H^{m+3+N^*-k}(S^{d-1}),\quad 0\le k\le N^*-2,\label{eq:subspace(1)}\\
&a_{N^*-1}\in H^{m+3}(S^{d-1})\setminus H^{m+4}(S^{d-1}),\label{eq:subspace(2)}\\
&a_{N^*}\in H^{m+3}(S^{d-1})\,.\label{eq:subspace(3)}
\end{align}
One then easily obtains from \eqref{eq:Lap(u)} that $\Delta u\in\AH_N^{m,p}\,,$
but $u\in\AH_N^{m,p}\setminus\AH_N^{m+2,p}$ since $a_{N^*-1}\notin H^{m+4}(S^{d-1})$.
In fact, let us define a different subspace of $\AH_N^{m,p}$:
\[
\widetilde\AH_N^{m+2,p}:=\Big\{u\in\AH_N^{m,p}\,:\,f\in H^{m+2,p}_N,
\eqref{eq:subspace(1)},\,\eqref{eq:subspace(3)}\,\,\text{hold},\,\,
\eqref{eq:subspace(2)}\,\,\text{is replaced by}\,\,a_{N^*-1}\in H^{m+3}
\Big\}\,.
\] 
Note that $\AH_N^{m+2,p}\subsetneqq\widetilde\AH_N^{m+2,p}$ and,
by arguing as in the proof of Proposition \ref{pr:SonAH}, one  sees that $\widetilde\AH_N^{m+2,p}$ is in the domain
of the generator $\Lambda$ of $\{S(t)\}_{t\ge 0}$ on $\AH_N^{m,p}$. This also shows that the domain of $\Lambda$  
cannot be $\AH_N^{m+2,p}$.

\section{Some basic nonlinear theory}\label{sec:mildsolutions}
The results in this section are fairly standard, but we collect them here and provide some proofs for 
completeness and convenience. Let us consider the equation 
\begin{equation}\label{eq:Lipschitz-parabolic}
\left\{
\begin{array}{l}
u_t=A u + F(t,u), \quad\hbox{for $0<t\le T$,} \\
u|_{t=0}=u_0,
\end{array}
\right.
\end{equation}
on a Banach space $X$. Here we assume $A$ is closed, densely defined and generates a strongly continuous 
semigroup $S(t)$ on $X$ that is quasibounded: for some $\omega>0$ we have
\begin{equation}\label{def:S-quasibounded}
\|S(t)v\|_X \le C\,e^{\omega t}\,\|v\|_X \quad\hbox{for}\ v\in X, \ t>0.
\end{equation}
First consider a {\em mild solution}, i.e.\ a continuous function $u:[0,T]\to X$ that 
satisfies the integral equation
\begin{equation}\label{eq:mild_solution}
u(t)=S(t)u_0+\int_0^t S(t-s)F(s,u(s))\,ds.
\end{equation}
As for conditions on $F$, let us assume that $F : [0,\infty)\times X\to X$ is continuous and {\em globally Lipschitz} on $X$ on 
any finite interval $[0,T]$, $T>0$. In other words, if we fix $T\in (0,\infty)$, then there exists $L\equiv L_T>0$ such that
\begin{equation}\label{global-Lipschitz}
\|F(t,u)-F(t,v)\|_X \le L\,\|u-v\|_X \quad \hbox{for}\ 0\le t\le T, \ u,v\in X.
\end{equation}
We have the following existence and uniqueness result (cf.\ Theorem 1.2 in Chapter 6 of \cite{Pazy}):

\begin{Prop}\label{pr:global-existence-uniqueness} 
If $A$ is a closed, densely defined operator on $X$ that generates a strongly continuous semigroup $S(t)$ satisfying 
\eqref{def:S-quasibounded} and $F:X\to X$ satisfies the globally Lipschitz condition \eqref{global-Lipschitz}, 
then for every $u_0\in X$ and $T>0$, \eqref{eq:Lipschitz-parabolic} admits a unique mild solution $u\in C([0,T],X)$.
\end{Prop}

\proof
Let $E:=C([0,T],X)$ and for $v\in E$ define
\[
\Phi\big(v\big)(t):=S(t)u_0+\int_0^t S(t-s)\,F(s,v(s))\,ds.
\]
We want to show $\Phi$ has a fixed point in $E$. Since $T<\infty$, we obtain from \eqref{def:S-quasibounded} that
there exists $C_T>0$ such that $\|S(t)\|_{X\to X}<C_T$ for $0\le t\le T$. In addition, let us consider the equivalent norm on $E$
\begin{equation}
\|v\|_E:=\sup\{e^{-kt}\,\|u(t)\|_X: 0 \le t\le T\},
\end{equation}
where $k>0$ is to be chosen (sufficiently large). In view of \eqref{global-Lipschitz},
\[
\begin{aligned}
\big\|\Phi\big(v\big)(t)\big\|_X & \le C\,\left(\|u_0\| +\int_0^t \Big(L\|v(s)\|_X+\|F(s,0)\|_X\Big)\,ds\right)\\
&\le C_1 + C_2\,\|v\|_E \quad \hbox{where}\ C_1, C_2 \ \hbox{may depend on $k$ and $T$.}
\end{aligned}
\]
Hence, $\Phi$ gives a well defined map $\Phi : E\to E$. To show $\Phi:E\to E$ is a contraction, we estimate
\[ 
\begin{aligned}
\|\Phi(u)-\Phi(v)\|_E& =  \sup \left\{e^{-kt}\,\left\|\int_0^t S(t-s)(F(s,u)-F(s,v))\,ds\right\|_X:0 \le t\le T\right\} \\
&\le C\,L\, \sup\left\{e^{-kt}\int_0^t e^{ks}\,ds: 0 \le t\le T \right\}\, \|u-v\|_E \\
&\le\,\frac{C\, L}{k}\, \|u-v\|_E.
\end{aligned}
\]
If $k$ is sufficiently large, $\Phi$ is a contraction on $E$. By the contraction mapping theorem, $\Phi:E\to E$ has 
a unique fixed point.
\endproof

Let us now assume that $F : [0,\infty)\times X\to X$ is continuous and that it satisfies the following Lipschitz condition on 
bounded sets of $X$: for any $R,T>0$ there exists $L_{R,T}>0$ such that
\begin{equation}\label{local-Lipschitz} 
\|F(t,u)-F(t,v)\|_X \le L_{R,T}\,\|u-v\|_X \quad \hbox{for}\ 0\le t\le T, \ \|u\|_X,\|v\|_X\le R.
\end{equation}
Then we have the following local existence and uniqueness result

\begin{Prop}\label{pr:local-existence-uniqueness} 
Suppose $A$ is a closed, densely defined operator on $X$ that generates a strongly continuous semigroup $\{S(t)\}_{t\ge 0}$ satisfying 
\eqref{def:S-quasibounded} and $F : [0,\infty)\times X\to X$ satisfies the Lipschitz condition \eqref{local-Lipschitz}. 
Then for every $u_0\in X$ there exists a maximal interval of existence $T_{\rm max}>0$ such that \eqref{eq:Lipschitz-parabolic}
admits a unique mild solution $u\in C([0,T_{\rm max}),X)$. In addition, if $T_{\rm max}<\infty$ then 
$\|u(t)\|_X\to\infty$ as $t\to T_{\rm max}$.
Moreover, if $F : [0,\infty)\times X\to X$ is continuously differentiable, then for any $u_0\in D_X(A)$ we have that
$u\in C([0,T_{\rm max}),D_X(A))\cap C^1([0,T_{\rm max}),X)$ is a solution of \eqref{eq:Lipschitz-parabolic}.
\end{Prop}

\noindent For a proof of this result, see the proof of Theorem 1.4 and Theorem 1.5 in Chapter 6 of \cite{Pazy}.
In the case when $A$ is a generator of an {\em analytic} semigroup and $F$ is independent of $t$ the mild solution in 
Proposition \ref{pr:local-existence-uniqueness} has additional regularity. More specifically, we have the following

\begin{Prop}\label{prop:analytic_semigroup_regularity} 
Suppose $A$ is a closed, densely defined operator on $X$ that generates an analytic semigroup $\{S(t)\}_{t\ge 0}$ satisfying 
\eqref{def:S-quasibounded} and $F : X\to X$ is independent of $t$ and satisfies the Lipschitz condition \eqref{local-Lipschitz} 
(with Lipschitz constant $L$ independent of $T$). Then, for every $u_0\in X$ the mild solution $u\in C([0,T_{\rm max}),X)$ from 
Proposition \ref{pr:local-existence-uniqueness} belongs to $C^1((0,T),X)$, $u(t)\in D_X(A)$ for any $t\in(0,T_{\rm max})$, 
and it satisfies \eqref{eq:Lipschitz-parabolic}.
\end{Prop}

\noindent This proposition follows by combining the first part of Proposition \ref{pr:local-existence-uniqueness},
Theorem 3.1 in Chapter 6 in \cite{Pazy}, and the following integral representation of the solution $u(t)$ for $0<t_0<t<T$,
\[
u(t)=S(t-t_0)u(t_0)+\int_0^{t-t_0}S(t-t_0-s) F\big(u(s)\big)\,ds.
\]

\medskip

Now, we make \eqref{eq:Lipschitz-parabolic} a little more specific and consider
\begin{equation}\label{eq:nonlinear_heat}
\left\{
\begin{array}{l}
u_t=\Lap u + F(u), \quad\hbox{for $t>0$, \ $x\in\R^d$,} \\
u|_{t=0}=u_0,
\end{array}
\right.
\end{equation}
as an evolution equation in $X=H^{m,p}_\delta\equiv H^{m,p}_\delta(\R^d)$.
Moreover, we want  solutions of \eqref{eq:Lipschitz-parabolic} such that
\begin{equation}\label{u-strict}
u\in C([0,T_{\rm max}),H^{m+2,p}_\delta)\cap C^1([0,T_{\rm max}),H^{m,p}_\delta).
\end{equation} 
The following theorem follows from Theorem 2.5.6 in \cite{Zheng}.

\begin{Prop}\label{pr:localexistence-H} 
For $1<p<\infty$, $m\in\Z_{\ge 0}$, and $\delta\in\R$, suppose $F : H^{m+2,p}_\delta\to H^{m+2,p}_\delta$
is continuously differentiable and  $u_0\in H^{m+2,p}_\delta$. Then the mild solution $u\in C([0,T),H^{m+2,p}_\delta)$ solution 
given by Proposition \ref{pr:local-existence-uniqueness} belongs to \eqref{u-strict} and satisfies \eqref{eq:nonlinear_heat}. 
\end{Prop}



\end{document}